\documentclass{article}
\usepackage[utf8]{inputenc}
\usepackage[T1]{fontenc}
\usepackage{amsmath, amssymb, amsthm,verbatim, hyperref, tikz}
\usepackage{amsfonts}
\usepackage{hyperref}
\usepackage[noabbrev]{cleveref}
\usepackage{caption}
\usepackage{graphicx}
\usepackage{authblk}
\usepackage{wrapfig}
\usepackage{subcaption}
\usepackage{indentfirst}
\usepackage{marginnote}
\usepackage{mathtools}
\usepackage{multicol}
\usepackage{soul}
\usepackage[document]{ragged2e}
\usepackage[left=3.00cm, right=3.00cm, top=3.00cm, bottom=3.00cm]{geometry}

\theoremstyle{plain}

\newtheorem{thm}{Theorem}[section]

\newtheorem{cor}[thm]{Corollary}
\newtheorem{lemma}[thm]{Lemma}

\theoremstyle{definition}
\newtheorem{definition}[thm]{Definition}

\theoremstyle{remark}

\newtheorem{remark}[thm]{Remark}

\title{
\vspace{-2em}
\textbf{Minimal Matchings for dP3 Cluster Variables}}
\author{Judy Hsin-Hui Chiang, Gregg Musiker, Son Nguyen}
\date{\vspace{-1em}}

\begin{document}

\maketitle

\begin{abstract}
    In previous work \cite{LM}, Tri Lai and the second author studied a family of subgraphs of the dP3 brane tiling, called Aztec castles, whose dimer partition functions provide combinatorial formulas for cluster variables resulting from mutations of the quiver associated with the del Pezzo surface dP3. In our paper, we investigate a variant of the dP3 quiver by considering a second alphabet of variables that breaks the symmetries of the relevant recurrences. This deformation is motivated by the theory of cluster algebras with principal coefficients introduced by Fomin and Zelevinsky. Our main result gives an explicit formula extending previously known generating functions for dP3 cluster variables by using Aztec castles and constructing their associated minimal matchings.    
\end{abstract}


\tableofcontents

\section{Introduction}\label{sec:intro}

    \justifying
    Upon Fomin and Zelevinsky's pioneering work \cite{FZ1,FZ4} in cluster algebras for the study of total positivity and dual canonical bases in semisimple Lie groups, a great variety of its applications have been found in combinatorics, tropical geometry \cite{SW}, Teichm\"uller theory \cite{FG}, and representation theory \cite{KB}. With the introduction of the Laurent phenomenon, mathematicians (\cite{S,LM,BMPW}) have been intrigued to study combinatorial interpretations for the cluster variables as perfect matchings of graphs, under suitable weighting schemes. Of particular interest is the situation where the graphs are directly related to the quiver of the cluster algebra, namely when they are subgraphs of the dual of the quiver.

    The main concern of our paper is the \textit{del Pezzo 3 (dP3) quiver} shown in Figure \ref{fig:dP3quiver}. Denote $\mu_i$ a mutation at vertex $i$. In this paper, we will consider the actions defined in Definition \ref{def:tau}:
    \begin{itemize}
        \centering
        \item[] $\tau_{1}=\mu_{1} \circ \mu_{2} \circ (12)$,
        \item[] $\tau_{2}=\mu_{3} \circ \mu_{4} \circ (34)$,
        \item[] $\tau_{3}=\mu_{5} \circ \mu_{6} \circ (56)$,
        \item[] $\tau_{4}= \mu_{1} \circ \mu_{4} \circ \mu_1 \circ \mu_5 \circ \mu_1 \circ (145)$,
        \item[] $\tau_{5}= \mu_{2} \circ \mu_{3} \circ \mu_2 \circ \mu_6 \circ \mu_2 \circ (236)$,
    \end{itemize}
    In \cite{Z}, Zhang proved an explicit formula for the cluster variables of the dP3 quiver under a sequence of mutations $\tau_1\tau_2\tau_3\tau_1\tau_2\tau_3\ldots$ as perfect matchings of Aztec Dragons. The authors of \cite{LMNT} then found a formula for more general sequences of mutations $\tau_{a_1}\tau_{a_2}\tau_{a_3}\ldots$ for $a_i\in\{1,2,3\}$. Finally, Lai and Musiker \cite{LM} found a formula for any sequence of mutations $\tau_{a_1}\tau_{a_2}\tau_{a_3}\ldots$ for $a_i\in\{1,2,3,4,5\}$ as perfect matchings of Aztec Castles with Aztec Dragons and the graphs of \cite{LMNT} being special cases.  Furthermore, the mutation sequences considered in \cite{LM} were general enough to obtain formulas for any cluster variable of the dP3 quiver reachable via a toric mutation sequence, see Section \ref{subsec:toricprism} for more details.
    
    \begin{figure}[h!]
        \centering
        \begin{minipage}{.5\textwidth}
            \centering
            \includegraphics[scale = 0.5]{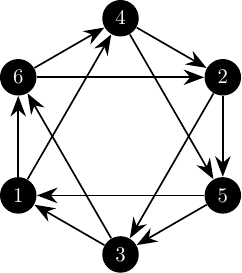}
            \caption{dP3 quiver}
            \label{fig:dP3quiver}
        \end{minipage}%
        \begin{minipage}{.5\textwidth}
            \centering
            \includegraphics[scale = 0.5]{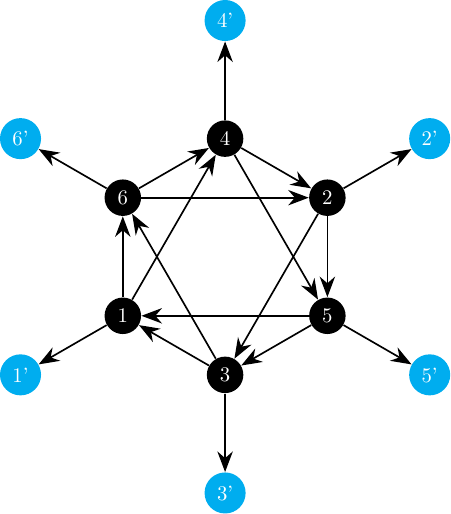}
            \caption{Framed dP3 quiver}
            \label{fig:dP3Framed}
        \end{minipage}
    \end{figure}

    In this paper, we will study a variant of the dP3 quiver where we introduce a second set of cluster variables by framing the quiver using new \textit{frozen vertices} as shown in Figure \ref{fig:dP3Framed}. This deformation is motivated by the theory of cluster algebras with principal coefficients introduced by Fomin and Zelvinsky in \cite{FZ4}. A natural question to ask is to find an explicit formula for this new quiver in terms of Aztec Castles. It turns out that it suffices to find a correct \textit{minimal matching} of the Aztec Castles. This is the main concern of this paper, and the construction of the minimal matching will be stated in Section \ref{subsec:construction}.

    \textbf{Contents.} The paper is outlined as follows. In Section \ref{sec:background}, we review the backgrounds of dP3 quivers, the construction of Aztec castles, and some preliminary results as well as proof techniques from previous work. Our main theorem is the construction of its minimal matching in Section \ref{sec:minCastle}, where we also give some examples. We leave our proof of the theorem to Section \ref{sec:proof}.

\section{Background}
\label{sec:background}

\subsection{Quiver and cluster mutations}\label{subsec:quivclus}

    \justifying
    A \textbf{quiver} $Q$ is a directed finite graph with a set of vertices $V$ and a set of directed edges $E$ connecting them such that there are no loops or $2$-cycles.
    We can relate a \textbf{cluster algebra} with \textbf{initial seed} $\{x_{1},x_{2},\ldots,x_{n}\}$ to $Q$ by associating a cluster variable $x_{i}$ to every vertex labeled $i$ in $Q$ where $|V| = n$.
    The \textbf{cluster} is the union of the cluster variables at each vertex.

    \begin{definition}\label{def:quivmut}[Quiver Mutation \cite{FZ1}]
        Mutating at a vertex $i$ in $Q$ is denoted by $\mu_{i}$ and corresponds to the following actions on the quiver:
        \begin{itemize}
        \item For every 2-path through $i$ (e.g. $j \rightarrow i \rightarrow k$), add an edge from $j$ to $k$.
        \item Reverse the directions of the arrows incident to $i$
        \item Delete any 2-cycles created from the previous two steps.
        \end{itemize}
    \end{definition}

    When we mutate at a vertex $i$, the cluster variable at this vertex is updated and all other cluster variables remain unchanged.
    The action of $\mu_{i}$ on the cluster leads to the following binomial exchange relation:
    \begin{equation*}\label{eq:exchange relation}
        x'_{i}x_{i} = \prod_{i \rightarrow j \; \mathrm{in} \; Q}x_{j}^{a_{i \rightarrow j}} + \prod_{j \rightarrow i \; \mathrm{in} \; Q}x_{j}^{b_{j \rightarrow i}}
    \end{equation*}
    where $x_i'$ is the new cluster variable at vertex $i$, $a_{i \rightarrow j}$ denotes the number of edges from $i$ to $j$, and $b_{j \rightarrow i}$ denotes the number of edges from $j$ to $i$.

    It was proved in \cite{FZ1} that every cluster variable is a \textit{Laurent polynomial} in $\mathbb{Z}[x_1^{\pm 1},\ldots, x_n^{\pm 1}]$, i.e.
    \[ x_m = \dfrac{P(x_1,\ldots,x_n)}{x_1^{d_1}\ldots x_n^{d_n}} \]
    for all $m$.

\subsection{del Pezzo 3 quiver and lattice}\label{subsec:dP3}

    The focus of this paper is the \textbf{del Pezzo 3 (dP3) quiver} illustrated in Figure~\ref{fig:dp3Quiver}. By unfolding this quiver, we get the infinite unfolded dP3 quiver as shown in Figure \ref{fig:unfoldeddP3}. Then, taking the dual graph of the unfolded quiver yields its brane tiling in Figure~\ref{fig:dP3Lattice}, which will be referred to as the \textbf{dP3 lattice}. This dP3 lattice is an example of a brane tiling, which are doubly periodic, bipartite, planar graphs that arise in string theory \cite{FHKV}. Using the dimer model on such graphs, theoretical physicists can associate an infinite class of supersymmetric quiver gauge theories to a corresponding toric variety (which is a Calabi–Yau 3-fold) as well as this combinatorial model. They appear physically in string theory through the intersections of NS5 and D5-branes which are dual to a configuration of D3-branes probing the singularity of a toric Calabi–Yau threefold. Because of its geometry connection and how the (3 + 1) dimensional supersymmetric gauge field theory lives on the world volume of the D3-brane, it can be represented by the \textbf{dP3 quiver}. In this special case of the dP3 lattice, these brane tilings were used by Cottrell-Young \cite{CY} as a version of the domino shuffling algorithm.
    
    \begin{figure}[h!]
        \centering
        \begin{minipage}[b]{.3\textwidth}
            \centering
            \includegraphics[scale = 0.5]{fig/dP3Quiver.pdf}
            \caption{dP3 quiver}
            \label{fig:dp3Quiver}
        \end{minipage}%
        \begin{minipage}[b]{.7\textwidth}
            \centering
            \includegraphics[scale = 0.5]{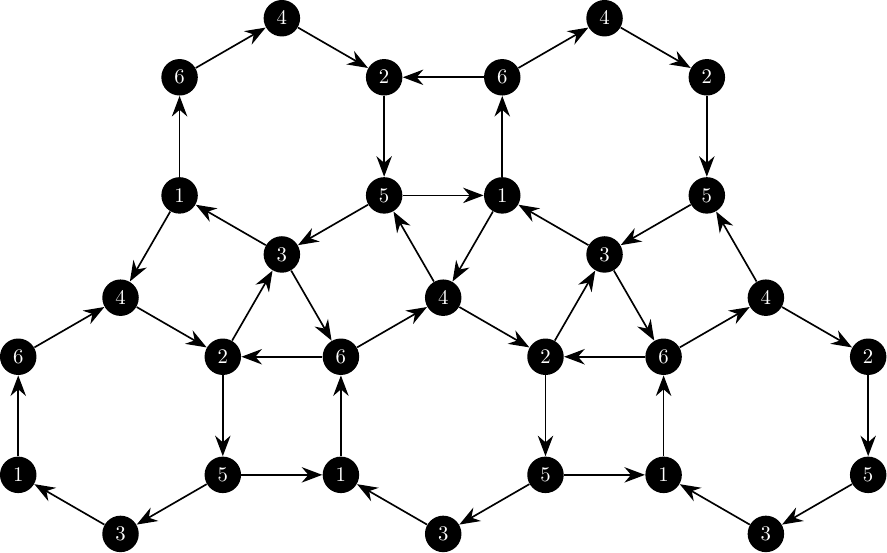}
            \caption{Unfolded dP3 quiver}
            \label{fig:unfoldeddP3}
        \end{minipage}
    \end{figure}
    
    \begin{figure}[h!]
        \centering
        \includegraphics[width=6cm]{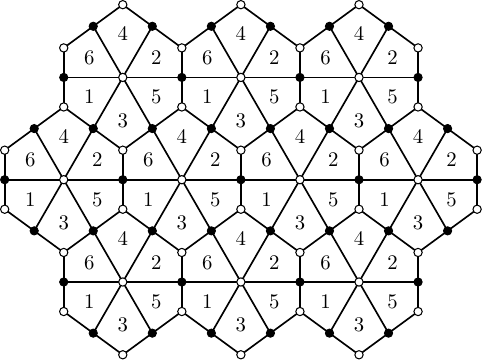}
        \caption{dP3 lattice}
        \label{fig:dP3Lattice}
    \end{figure}

    In \cite{LM}, Lai and Musiker studied sequences of \textit{toric mutations} on the dP3 quiver (see Section \ref{subsec:toricprism}). They found an expression for resulting cluster variables in terms of perfect matchings of \textit{Aztec Castles} (see Section \ref{subsec:castle}).

\subsection{Toric mutations and prism walk}\label{subsec:toricprism}

    \subsubsection{Toric mutations}\label{subsubsec:toricMutation}

    A vertex is \textbf{toric} if its in-degree and out-degree are both $2$. A \textbf{toric mutation} is a mutation at a toric vertex. In this paper, we will study the following five actions on the dP3 quiver, which are also the main actions studied in \cite{LM}.

    \begin{definition}\label{def:tau}
        Define the following actions 
        \begin{itemize}
        \centering
        \item[] $\tau_{1}=\mu_{1} \circ \mu_{2} \circ (12)$,
        \item[] $\tau_{2}=\mu_{3} \circ \mu_{4} \circ (34)$,
        \item[] $\tau_{3}=\mu_{5} \circ \mu_{6} \circ (56)$,
        \item[] $\tau_{4}= \mu_{1} \circ \mu_{4} \circ \mu_1 \circ \mu_5 \circ \mu_1 \circ (145)$,
        \item[] $\tau_{5}= \mu_{2} \circ \mu_{3} \circ \mu_2 \circ \mu_6 \circ \mu_2 \circ (236)$,
        \end{itemize}
        where we apply a graph automorphism of $Q$ and permutation to the labeled seed after the sequence of mutations.
    \end{definition}

    One can then check that on the level of quivers and labeled seeds (i.e. ordered clusters), we have the following identities, which are also noted in \cite{LM}:
    For all $i,j$ such that $1 \leq i \not = j \leq 3$, we have
    
    \begin{equation*}
    \label{eq:tau_relations}
    \begin{split}
     \tau_1(Q) = \tau_2(Q) = \tau_3(Q) &= \tau_4(Q) = \tau_5(Q) = Q  \\
    (\tau_{i})^{2} \{x_1,x_2\dots, x_6\} = (\tau_{4})^{2} \{x_1,x_2\dots, x_6\} &= (\tau_{5})^{2} \{x_1,x_2\dots, x_6\} = \{x_1,x_2\dots, x_6\} \\
    (\tau_{i}\tau_{j})^{3} \{x_1,x_2\dots, x_6\}&= \{x_1,x_2\dots, x_6\}, \\
    \tau_i \tau_4 \{x_1,x_2\dots, x_6\} &= \tau_4 \tau_i \{x_1,x_2\dots, x_6\}, \\
    \tau_i \tau_5 \{x_1,x_2\dots, x_6\} &= \tau_5 \tau_i \{x_1,x_2\dots, x_6\}.
    \end{split}
    \end{equation*}

    \subsubsection{Prism walk}\label{subsubsec:prismWalk}

    We will model the mutations defined in Definition \ref{def:tau} as prism walk on a square triangulated lattice of $\mathbb{Z}^3$, such that a two-dimensional cross-section is illustrated in Figure \ref{fig:lattice}. We will place the prism so that the coordinates of vertices $1,\ldots,6$ are $(0,-1,1),(0,-1,0),(-1,0,0),(-1,0,1),(0,0,1),(0,0,0)$ respectively. The reason is the cluster variables corresponding to these coordinates, described in Section \ref{subsec:castle}, are $x_1,\ldots,x_6$ respectively.

    \begin{figure}[h!]
        \centering
        \begin{subfigure}[b]{0.3\textwidth}
            \centering
            \includegraphics[width = 0.5\textwidth]{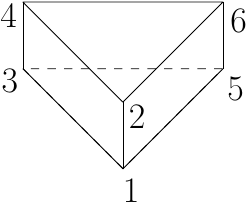}
            \caption{Prism}
            \label{fig:prism}
        \end{subfigure}
        \begin{subfigure}[b]{0.4\textwidth}
            \centering
            \includegraphics[width = \textwidth]{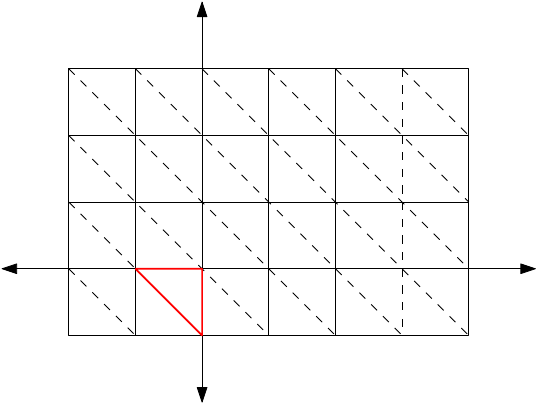}
            \caption{Lattice}
            \label{fig:lattice}
        \end{subfigure}
        \caption{Prism and lattice}
        \label{fig:prims-lattice}
    \end{figure}

    We define the action of $\mu_i$ on the prism as follows. First, $\mu_1$ (resp. $\mu_2$) reflects vertex $1$ (resp. $2$) of the prism about the center of the rectangle formed by vertices $3,4,5,6$. Similarly, $\mu_3$ (resp. $\mu_4$) reflects vertex $3$ (resp. $4$) of the prism about the center of the rectangle formed by vertices $1,2,5,6$, and $\mu_5$ (resp. $\mu_6$) reflects vertex $5$ (resp. $6$) of the prism about the center of the rectangle formed by vertices $1,2,3,4$.

    As a result, the action of $\tau_1$ on the prism can be visualized as in Figure \ref{fig:tau1OnPrism}. Thus, $\tau_1$ acts on the prism by reflecting the edge $12$ about the rectangle formed by vertices $3,4,5,6$. Similarly, $\tau_2$ acts on the prism by reflecting the edge $34$ about the rectangle formed by vertices $1,2,5,6$, and $\tau_3$ acts on the prism by reflecting the edge $56$ about the rectangle formed by vertices $1,2,3,4$.
    
    \begin{figure}[h!]
        \centering
        \begin{minipage}{.5\textwidth}
            \centering
            \includegraphics[width=6cm]{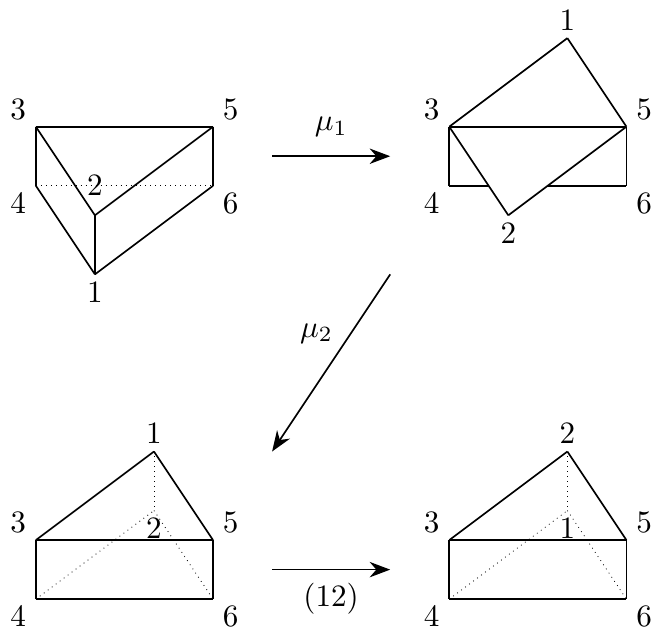}
            \caption{$\tau_1$ action on the prism}
            \label{fig:tau1OnPrism}
        \end{minipage}%
        \begin{minipage}{.5\textwidth}
            \centering
            \includegraphics[width=6cm]{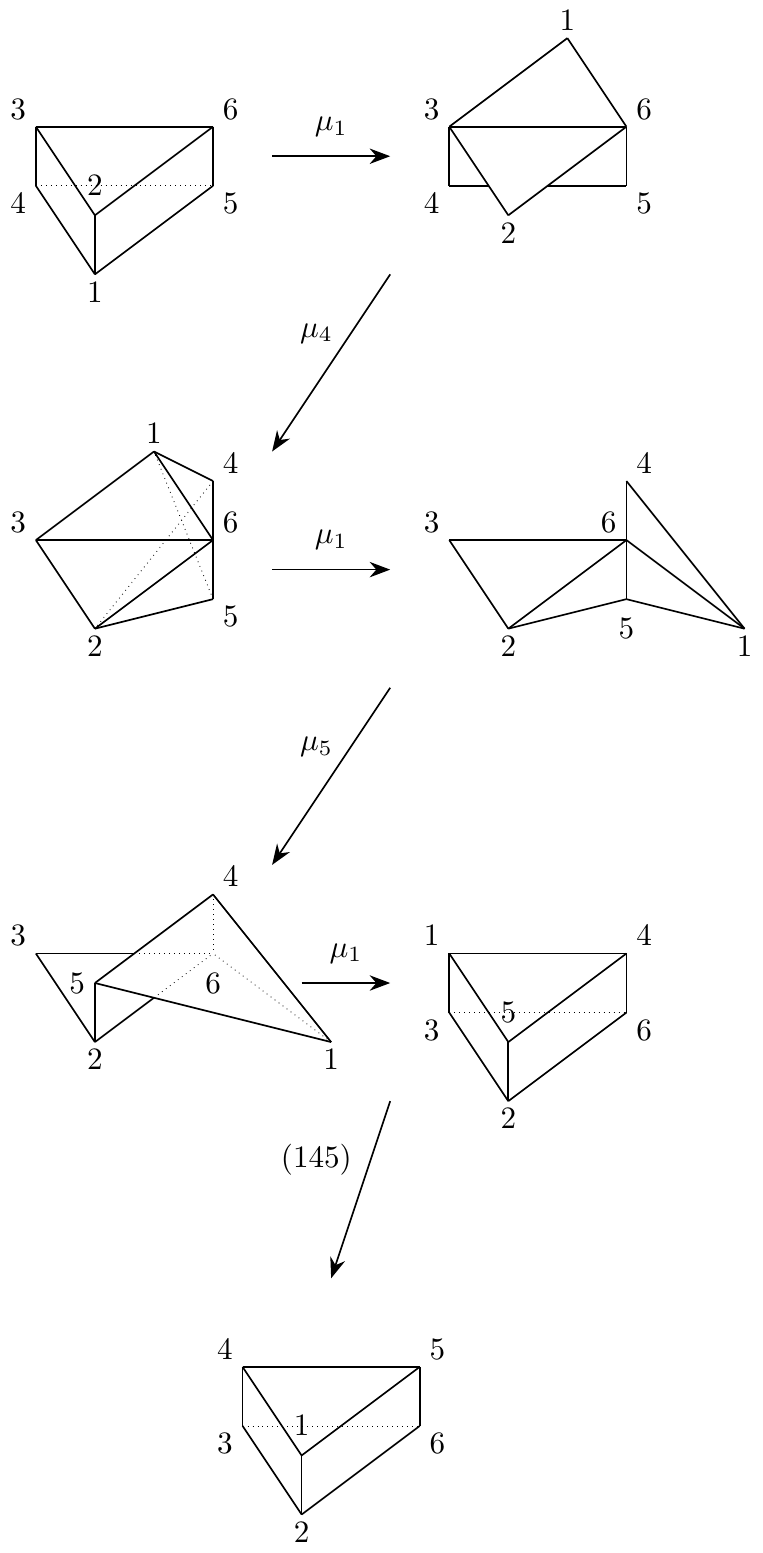}
            \caption{$\tau_4$ action on the prism}
            \label{fig:tau4OnPrism}
        \end{minipage}
    \end{figure}

    Similarly, Figure \ref{fig:tau4OnPrism} illustrates the action of $\tau_4$ on the prism. We can see that $\tau_4$ reflects the triangle formed by vertices $1,4,5$ about the triangle formed by the other vertices. Similarly, $\tau_5$ reflects the triangle formed by vertices $2,3,6$ about the triangle formed by the other vertices.

    In summary, the actions of $\tau_1$, $\tau_2$, and $\tau_3$ change the $x$ and $y$ coordinates of the prism while fixing the $z$ coordinate. This is the cross-section illustrated in Figure \ref{fig:lattice}.  The actions of $\tau_4$ and $\tau_5$ change the $z$ coordinate of the prism while fixing the $x$ and $y$ coordinates. One can check that using these five $\tau$-mutations, we can indeed move the original prism to any isometric prism in the $\mathbb{Z}^3$ lattice. Conversely, any sequence of $\tau$-mutations can be modeled as a prism walk in the $\mathbb{Z}^3$ lattice. As a result, we can associate every cluster variable $z_{i,j,k}$ to a point $(i,j,k)$ in $\mathbb{Z}^3$ and vice versa. In the next section, we will also construct an Aztec Castle $\mathcal{C}_{i,j,k}$ for each point $(i,j,k)$. For \textit{nice} Aztec Castles, we can express $z_{i,j,k}$ in terms of perfect matchings of $\mathcal{C}_{i,j,k}$ (see Theorem \ref{thm:LMCastle}).

\subsection{Aztec Castles}\label{subsec:castle}
    \subsubsection{Construction}\label{subsubsec:castleConstruction}

    In this section, we review the construction of \textbf{Aztec Castles} in \cite{LM}. In general, the construction of Aztec Castles consists of the following steps.

    \begin{itemize}\label{itm:construction}
        \item \textit{Step 1:} We associate each point $(i,j,k)\in\mathbb{Z}^3$ with a 6-tuple
        \[ (j+k, -i-j-k, i+k, j+1-k, -i-j-1+k, i+1-k). \]
        We will explain this 6-tuple later in this section.
        \item \textit{Step 2:} We draw a \textbf{(six-sided) contour} $\mathcal{C}(j+k, -i-j-k, i+k, j+1-k, -i-j-1+k, i+1-k)$ on the dP3 lattice in the direction in Figure~\ref{fig:contourDirection}. We start from a vertex in the center of a hexagon and define the unit length to be two ``long'' edges of the lattice. Note that if an element of the tuple is $0$, we simply skip the corresponding side, and if an element is negative, we traverse in the opposite direction.

        \begin{figure}[h!]
            \centering
            \includegraphics[scale=0.6]{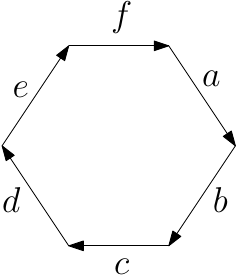}
            \caption{Contour direction}
            \label{fig:contourDirection}
        \end{figure}

        \item \textit{Step 3:} We remove every vertex outside the contour and keep only the vertices inside.

        \item \textit{Step 4:} We remove vertices along the sides as follows. For any side of positive (resp. negative) length, we remove all black (resp. white) vertices along that side. This side corresponds to a single vertex for any side of length zero. If any of the adjacent sides is negative, then this vertex is already removed. If this side is between two sides of length zero, we will also remove this vertex. The only case in which we keep this vertex is when it is between two sides of positive lengths.

        \item \textit{Step 5:} Finally, we have some "dangling" edges, which are edges in which one of the two incident vertices has degree $1$. These are the red edges in Figure~\ref{fig:castleConstructionStep4}. For these edges, we can either keep or remove the two incident vertices. The reason is that when considering perfect matchings of this graph, these edges are always forced to be in the matching, and they do not contribute to the weight of the matching (which will be defined in Section~\ref{subsubsec:castleWeight}). For most of this paper, we opt to keep these edges and call the resulting graph $\mathcal{C}_{i,j,k}$. 
        
    \end{itemize}

    Figure \ref{fig:castleConstruction} illustrates the construction of an Aztec Castle starting from the contour $\mathcal{C}(4,-3,0,3,-2,-1)$. The red point in Figure \ref{fig:castleConstructionStep2} marks the starting point of the contour.

    \begin{figure}[h!]
        \centering
        \begin{subfigure}{.32\textwidth}
            \centering
            \includegraphics[scale = 0.4]{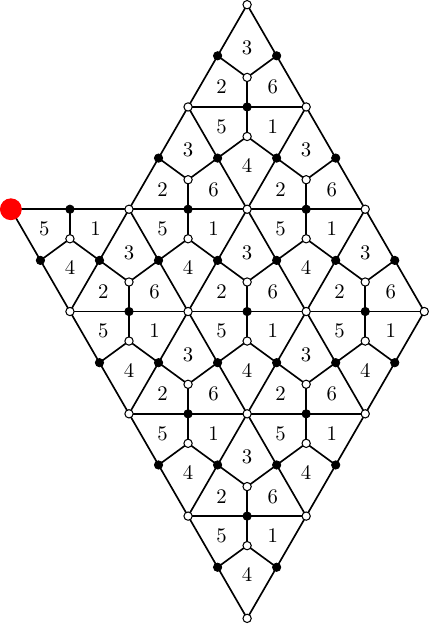}
            \caption{Step 3}
            \label{fig:castleConstructionStep2}
        \end{subfigure} \hspace{-5em}
        \begin{subfigure}{.32\textwidth}
            \centering
            \includegraphics[scale = 0.4]{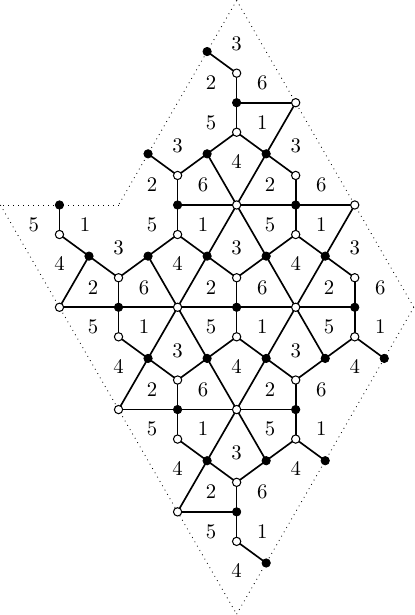}
            \caption{Step 4}
            \label{fig:castleConstructionStep3}
        \end{subfigure} \hspace{-5em}
        \begin{subfigure}{.32\textwidth}
            \centering
            \includegraphics[scale = 0.4]{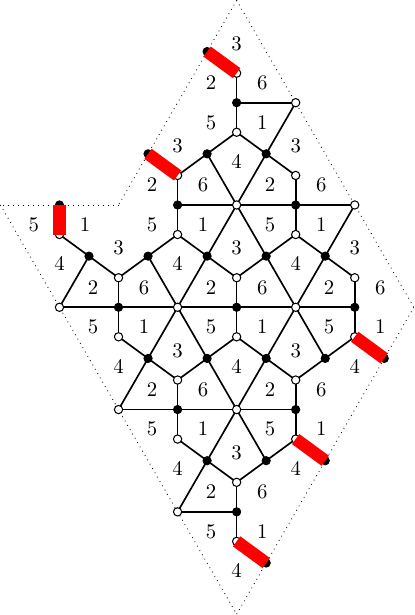}
            \caption{Step 5}
            \label{fig:castleConstructionStep4}
        \end{subfigure} \hspace{-5em}
        \begin{subfigure}{.32\textwidth}
            \centering
            \includegraphics[scale = 0.4]{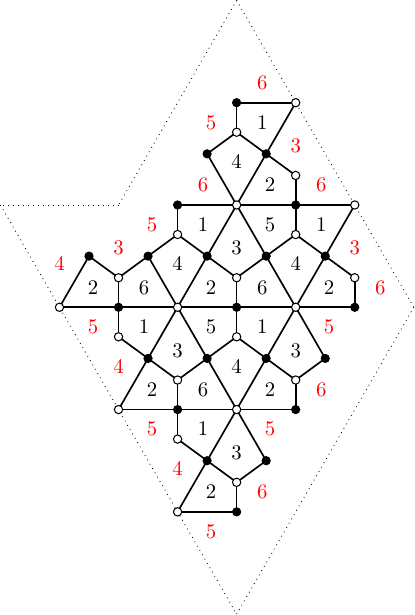}
            \caption{$\Hat{\mathcal{C}}_{i,j,k}$}
            \label{fig:castleConstructionStep5}
        \end{subfigure}
        \caption{Aztec Castle construction}
        \label{fig:castleConstruction}        
    \end{figure}

    Now, we introduce some terminologies that will be used in Section \ref{subsubsec:castleWeight} when discussing the weight of a perfect matching of Aztec Castles. Since the dangling edges do not contribute to the weight, we delete all dangling edges and subsequently all faces that do not have any incident edges left. We define the graph after such deletion $\Hat{\mathcal{C}}_{i,j,k}$. Also, we can define \textit{interior} of the graph, denoted as $\mathcal{C}^\circ_{i,j,k}$, to be the set of faces that have four edges in $C_{i,j,k}$. We define the \textit{boundary} of the graph, denoted as $\partial\mathcal{C}_{i,j,k}$, to be the set of faces with fewer than four edges. For example, Figure \ref{fig:castleConstructionStep5} shows the faces of $\Hat{\mathcal{C}}_{i,j,k}$. The faces in $\partial\mathcal{C}_{i,j,k}$ are colored red, and the remaining black faces are in $\mathcal{C}^\circ_{i,j,k}$.

    We now explain the 6-tuple $(a,b,c,d,e,f)$ in Step 1, as described in Lemma 5.3 of \cite{LM}. First of all, for the contour to be closed, we want
    \[
    a+b=d+e \quad \text{and} \quad c+d=f+a.
    \]
    We also want $b+c=e+f$, but this is implied by the above two relations, so we do not include this condition. Finally, since we will work with perfect matchings of this graph, we want the same number of white vertices and black vertices. By counting the number of vertices deleted on each side in step 3 of the construction, Lai and Musiker introduced a third condition which allows for an equal number of black and white vertices:
    \[
    a+b+c+d+e+f=1.
    \]
    One can check that the tuple
    \[ (j+k, -i-j-k, i+k, j+1-k, -i-j-1+k, i+1-k) \]
    satisfies all three aforementioned conditions.

    Figure \ref{fig:allRegionPos} shows the possible Aztec Castle shapes when $k\geq 1$. There are Aztec Castles from contours that have self-intersections. However, in this paper, we do not consider such Aztec Castles. The reason is apparent from Theorem \ref{thm:LMCastle}. Hence, in this paper, when we say Aztec Castles, we assume no self-intersections, i.e. we will only consider points outside the yellow region in Figure \ref{fig:allRegionPos}.

    \begin{figure}[h!]
        \centering
        \includegraphics[scale = 0.25]{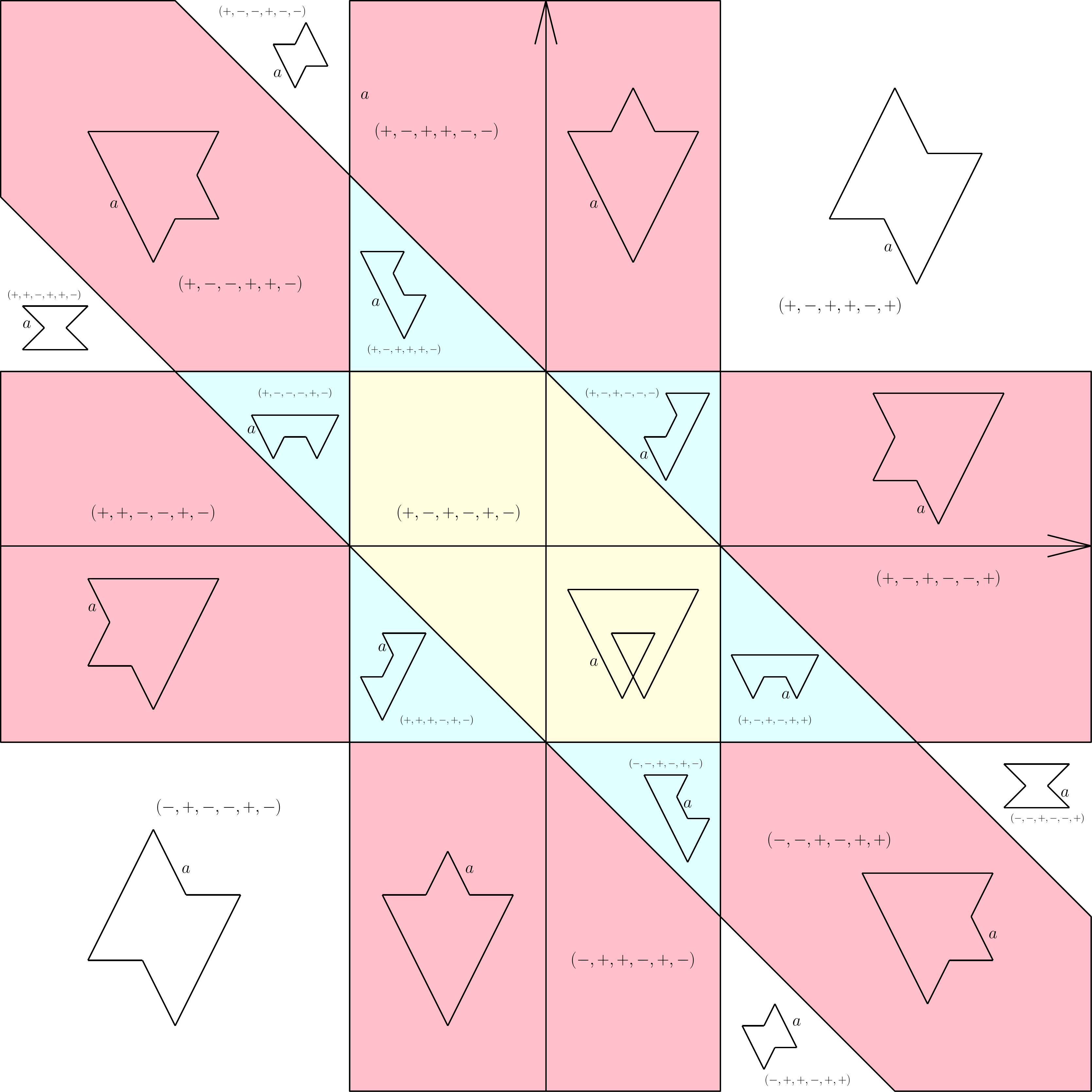}
        \caption{Possible shapes of Aztec Castles when $k\geq 1$}
        \label{fig:allRegionPos}
    \end{figure}

    \begin{remark}\label{rem:sign-pattern}
        Figure \ref{fig:allRegionPos} only shows the possible Aztec Castle shapes when $k\geq 1$. When $k\leq 0$, the sign patterns differ by a cyclic shift of three indices, i.e. $(s_1,s_2,s_3,s_4,s_5,s_6)\rightarrow(s_4,s_5,s_6,s_1,s_2,s_3)$. For example, the sign pattern of the top middle pink region when $k\geq 1$ is $(+,-,+,+,-,-)$ as in Figure \ref{fig:allRegionPos}. Then, when $k\leq 0$, the sign pattern of this region is $(+,-,-,+,-,+)$. Observe that the shapes in the white region actually stay the same under this cyclic shift.
    \end{remark}

    \begin{remark}\label{rem:sign-change}
        From Figure \ref{fig:allRegionPos}, we have the sign patterns break into orbits under rotations as follows. In the white regions, the sign patterns are $(+,-,+,+,-,+)$, $(-,+,-,-,+,-)$, and their cyclic rotations. In the pink regions, the sign patterns are $(+,+,-,-,+,-)$, $(-,-,+,+,-,+)$, and their cyclic rotations. In the blue regions, the sign patterns are $(+,+,+,-,+,-)$, $(-,-,-,+,-,+)$, and their cyclic rotations. Note that it suffices to discuss only the $k\ge 1$ cases since the $k\le 0$ cases will follow using Remark \ref{rem:sign-pattern}. This is because, for each region, a sign pattern will include exactly half of the subregions up to cyclic rotations. For example, in the pink region, three cyclic rotations of $(+,+,-,-,+,-)$ and three cyclic rotations of $(-,-,+,+,-,+)$ each show up as subregions for $k\ge 1$ cases and the rest for $k\le 0$. Similar results occur for the blue regions, as well as for the white regions (except there are only three cyclic rotations of each representative in the latter case).  Hence, up to cyclic rotation, then it is sufficient to only consider the $k\ge 1$ case for our purpose.
    \end{remark}
    
    \begin{remark} \label{rem:sign-pattern-4}
    By the explicit description of sign patterns given by Remark \ref{rem:sign-change}, it follows that for generic non-self-intersecting Aztec Castles, when traversing around the contours, the signs of the sides change exactly four times.
    \end{remark}
    \subsubsection{Weight}\label{subsubsec:castleWeight}

    For every Aztec Castle $\mathcal{C}_{i,j,k}$, we will use the common definition of the weight of a perfect matching $m$ as defined by Speyer in \cite{S}. Here we write it as Speyer expresses it for a general bipartite planar graph $G$. Fix a perfect matching $m$, and for a face $f$ of a $2s$-gon in $G$, let
    \begin{align*}
        \epsilon(f) = \begin{cases}
         (s-1)-|E(f) \cap m| &\text{ if }f\in G^{\circ},\\
        \lfloor\frac{s}{2}\rfloor-|E(f) \cap m|  &\text{ if }f\in \partial G.
    \end{cases}
    \end{align*}
    Recall that $G^{\circ}$ denotes the set of interior faces of $G$, and $\partial G$ denotes the set of boundary faces. Then, the weight of $m$ is defined to be
    \[ wt(m) = \prod_{f \in \hat{G}}x_f^{\epsilon(f)}. \]
    In our discussion of Aztec Castles, we let $G$ be the graph $\mathcal{C}_{i,j,k}$, $\hat{G}$ to be defined by deleting the dangling edges in Step 5 (see Figure \ref{fig:castleConstructionStep5}). Also, since all faces are $4$-gons, the formula for the weight can be simplified to
     \[
         wt(m)=\prod_{f\in \hat{G}} x_f^{(1-|E(f) \cap m|)}.
     \]
    With the weight of each perfect matching defined, we have the following definition of the weight of an Aztec Castle:
     \[
         wt(\mathcal{C}_{i,j,k}) = \sum_{m}wt(m).
     \]
    Lai and Musiker proved the following theorem about Aztec Castles.

     \begin{thm}[\cite{LM}]
         \label{thm:LMCastle}
         Let $z_{i,j,k}$ be the cluster variable at point $(i,j,k)$. Then if $\mathcal{C}(i,j,k)$ does not have self-intersections, we have

         $$z_{i,j,k}=wt(\mathcal{C}_{i,j,k}).$$
     \end{thm}

\subsection{Framed quiver and minimal matching}\label{subsec:framed}

\subsubsection{Framed quiver}

    For a quiver $Q$, the associated \textbf{framed quiver} $\hat{Q}$ is a directed graph in which
    \[
    V_{\hat{Q}} = V_Q~ \cup~\{v_{i+n}\mid v_i\in V_Q\} \quad \text{and} \quad E_{\hat{Q}} = E_Q~\cup~\{v_i \rightarrow v_{i+n}\mid v_i\in V_Q\}.
    \]
    For instance, the framed quiver of the dP3 quiver in Figure \ref{fig:dP3quiver} is the quiver in Figure \ref{fig:dP3Framed}. The additional vertices in $\hat{Q}$ are called \textit{frozen vertices}, which means that we never mutate at these vertices. We also associate new cluster variables $\{y_1,\ldots,y_n\}$ to the frozen vertices. Then, every cluster variable for this framed quiver is a Laurent polynomial in $\mathbb{Z}[x_1^{\pm 1},\ldots, x_n^{\pm 1}, y_1,\ldots,y_n]$, i.e. for all cluster variable $\hat{z}$, one can write
    \[\hat{z} = \dfrac{P(x_1,\ldots,x_n,y_1,\ldots,y_n)}{x_1^{d_1}\ldots x_n^{d_n}}.\]
    Following \cite[Theorem 3.7, Conjecture 5.4]{FZ4} as well as \cite[Theorem 1.7, Prop 3.1]{DWZ}, there is a unique term in this $\mathbb{Z}[x_1^{\pm 1},\ldots, x_n^{\pm 1}, y_1,\ldots,y_n]$ expansion for $\hat{z}$ with no $y_i$'s occuring as a factor.
    
    Given Theorem \ref{thm:LMCastle}, the cluster variables $z_{i,j,k}$ are expected to be generating functions (or termed weighted sums) over perfect matchings of the corresponding Aztec Castles $\mathcal{C}_{i,j,k}$. In this case, to account for the extra $y_i$'s, instead of simply taking the weight defined in Section \ref{subsubsec:castleWeight}, we also need to introduce the notion of \textbf{height} for each perfect matching.
    
    \begin{definition}\label{def:height}
        For each perfect matching $m$, the \textbf{height} of $m$ is ${\sf ht}(m) =f$, where $f$ is the face in a closed loop created by superimposing $m$ with the minimal matching, where the minimal matching is defined as the unique matching whose weight (in $x_i$'s) agrees with the unique term in $\hat{z}$ without a $y_i$ in it.    
    \end{definition}
    
    As suggested by Definition \ref{def:height}, our main question is: what is the \textit{correct minimal matching} of Aztec Castles?

    Fortunately, by Musiker--Schiffler (\cite[Remark 5.3]{musiker2010cluster}), the minimal matching of Aztec Castles $\mathcal{C}_{i,j,k}$ is the same as the minimal element of the twist down lattice of $\mathcal{C}_{i,j,k}$ studied by Propp in \cite{propp2002lattice}.  See also \cite{MMSBV} for a more recent treatment for general bipartite plane graphs.

\subsubsection{Twist down lattice of perfect matchings}\label{subsubsec:twist-down}
    
    Given graph $\mathcal{C}_{i,j,k}$ (defined as an Aztec Castle with the dangling edges removed), we may temporarily forget the labeling of faces in $\partial\mathcal{C}_{i,j,k}$ and consider the exterior to be a single unbounded face which we denote as $f^*$. A \textbf{face twist} on a perfect matching is the operation of removing the edges that form an alternating cycle of a face and inserting the complementary edges. We call a face \textbf{twistable} if a face twist can be applied to it. In other words, a $2s$-gon face is twistable if it has $s$ edges in the perfect matching, and a face twist on this face will replace the $s$ edges by the other $s$ edges.
    
    We say a twistable face of the Aztec Castle is \textbf{positive} (resp. \textbf{negative}) if the edges in the face when directed from black vertices to white vertices, circle the face in a counterclockwise (resp. clockwise) direction. A \textbf{twist down} is a face twist that converts a positive face to a negative face, and a \textbf{twist up} is a face twist that converts a negative face to a positive face. This face twist operation gives a distributive lattice on the set of perfect matchings of an Aztec Castle. 

    \begin{thm}[{\cite[Theorem 2]{propp2002lattice}}]\label{thm:face-twist-lattice}
        Let $\mathcal{M}$ be the (non-empty) set of perfect matchings of an Aztec Castle $\mathcal{C}_{i,j,k}$. If we say that one perfect matching $M$ covers another perfect matching $N$ exactly when $N$ is obtained from $M$ by twisting down at a face other than $f^*$, then the covering relation makes $\mathcal{M}$ into a distributive lattice.
    \end{thm}

    Figure \ref{fig:twistDownLattice} gives an example of Theorem \ref{thm:face-twist-lattice}. Furthermore, the unique minimal element of this lattice is the unique perfect matching in which no twisting down at a face other than $f^*$ is possible. In other words, every possible face twist at a face other than $f^*$ is at a negative face. Hence, to find the correct minimal matchings of Aztec Castles, it suffices to check this condition.

    \begin{cor}\label{cor:twist-condition}
        If a perfect matching of an Aztec Castle has no positive twistable face (except the unbounded face), then it is the minimal matching.
    \end{cor}

    \begin{figure}[h!]
        \centering
        \includegraphics[scale = 0.7]{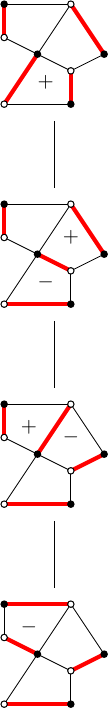}
        \caption{Twist down lattice of an Aztec Castle}
        \label{fig:twistDownLattice}
    \end{figure}

\section{Minimal matching of Aztec Castles}\label{sec:minCastle}

\subsection{Construction}\label{subsec:construction}

    In this section, we will give the construction of the minimal matching, which will be proved later in Section \ref{sec:proof}. Specifically, we will give the construction for generic cases when none of the sides is $0$. When one or more sides equal $0$, we give them arbitrary signs. The construction has two main steps.

    \begin{itemize}
        \item \textit{Step 1: Dividing the Castle into sectors.} We traverse along the sides of the contour in a clockwise direction. At each corner, we perform one of the following actions:
        \begin{itemize}
            \item If we move from a positive side to a positive side, draw a straight line in the direction of the second side.
            \item If we move from a negative side to a negative side, draw a straight line in the direction of the first side.
            \item If we move from a negative side to a positive side, draw a staircase diagonally, with the first step lying on the positive side.
            \item If we move from a positive side to a negative side, no action is required.
        \end{itemize}
        By Remark \ref{rem:sign-pattern-4}, we move from a negative side to a positive side exactly twice and move between two sides with the same sign exactly twice. Hence, there are exactly two straight lines and two staircases. 
        
        Furthermore, while it is not immediate from the construction, these two straight lines and two staircases intersect at two points (one per pair of line and staircase), and the two points can be connected by a straight line on the lattice. We connect these two points by that straight line, and we call this line the \textbf{zero line}. This line is parallel to two opposing sides of the contour (see Remark \ref{rem:zero-line-direct}). We will show the existence of this zero line in Section \ref{subsec:zero-line}.
        
        After this step, the Castle is divided into four sectors, two of them are each incident to one side of the contour while the other two are each incident to two sides.
        
        \item \textit{Step 2: Covering each sector according to the side.} We will use a universal covering for each sector, and the covering is determined by the side of the contour that the sector is incident to as in Table~\ref{table:covering}. Here, for example, by ``$1-4$'' we mean the edge between face $1$ and $4$.
        
        \begin{table}[h!]
            \centering
            \begin{tabular}{ |c|c|c| } 
             \hline
             Side & Positive & Negative \\ 
             \hline
             \hline
             $a$ & $1-4,2-5,3-6$ & $1-5,2-4,3-6$ \\
             \hline
             $b$ & $1-4,2-6,3-5$ & $1-4,2-5,3-6$ \\
             \hline
             $c$ & $1-3,2-6,4-5$ & $1-4,2-6,3-5$ \\
             \hline
             $d$ & $1-6,2-3,4-5$ & $1-3,2-6,4-5$ \\
             \hline
             $e$ & $1-5,2-3,4-6$ & $1-6,2-3,4-5$ \\
             \hline
             $f$ & $1-5,2-4,3-6$ & $1-5,2-3,4-6$ \\
             \hline
            \end{tabular}
            \caption{Universal covering for each case}
            \label{table:covering}
        \end{table} 
    \end{itemize}

    Figure \ref{fig:ConsExample} shows an example of this construction. The contour is $\mathcal{C}(5,-9,6,2,-6,3)$, giving the Castle $\mathcal{C}_{4,3,2}$. In Figure \ref{subfig:ConsDivide}, from side $a$ to $b$, and from side $d$ to $e$, we move from positive sides to negative sides, so we do nothing. From side $b$ to $c$, and from side $e$ to $f$, we move from negative sides to positive sides, so we draw staircases (\textcolor{violet}{colored purple}) with the first steps lying on the positive sides $c$ and $f$. From side $c$ to $d$, and from side $f$ to $a$, we move from positive sides to positive sides, so we draw straight lines (\textcolor{red}{colored red}) in the direction of sides $a$ and $d$. Furthermore, the two straight lines and staircases meet at two points that can be connected by another straight \textcolor{blue}{blue line}. (In Section \ref{subsec:zero-line}, we will call the blue line the zero line and show its existence). Finally, in Figure \ref{subfig:ConsMatching}, we cover each sector according to Table \ref{table:covering}. We refer the readers to Appendix \ref{append:examples} for more examples.

    \begin{figure}[h!]
    \centering
        \begin{subfigure}[b]{0.4\textwidth}
            \centering
            \includegraphics[width = \textwidth]{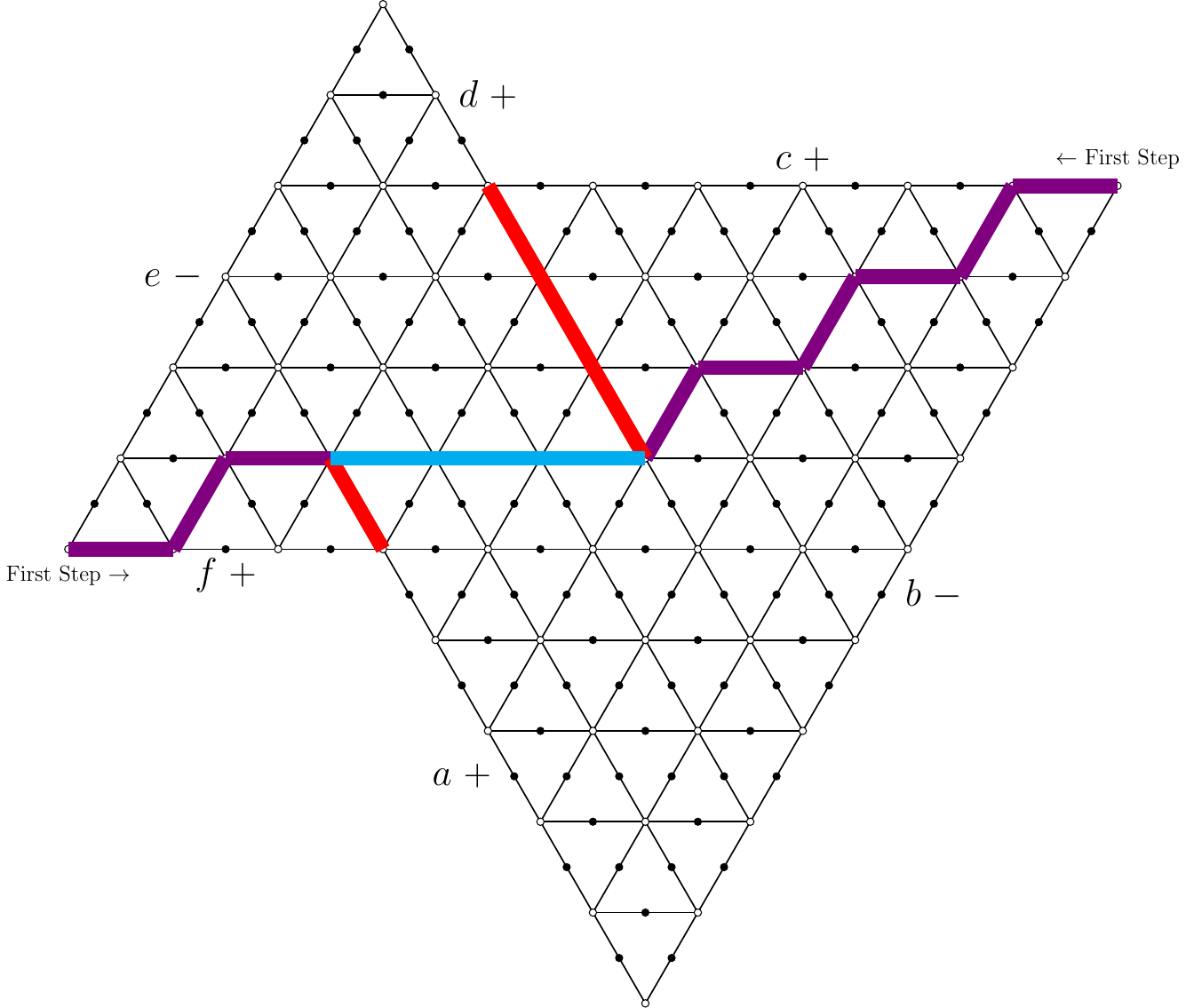}
            \caption{Four sectors}
            \label{subfig:ConsDivide}
        \end{subfigure}
     \quad
        \begin{subfigure}[b]{0.4\textwidth}
            \centering
            \includegraphics[width = \textwidth]{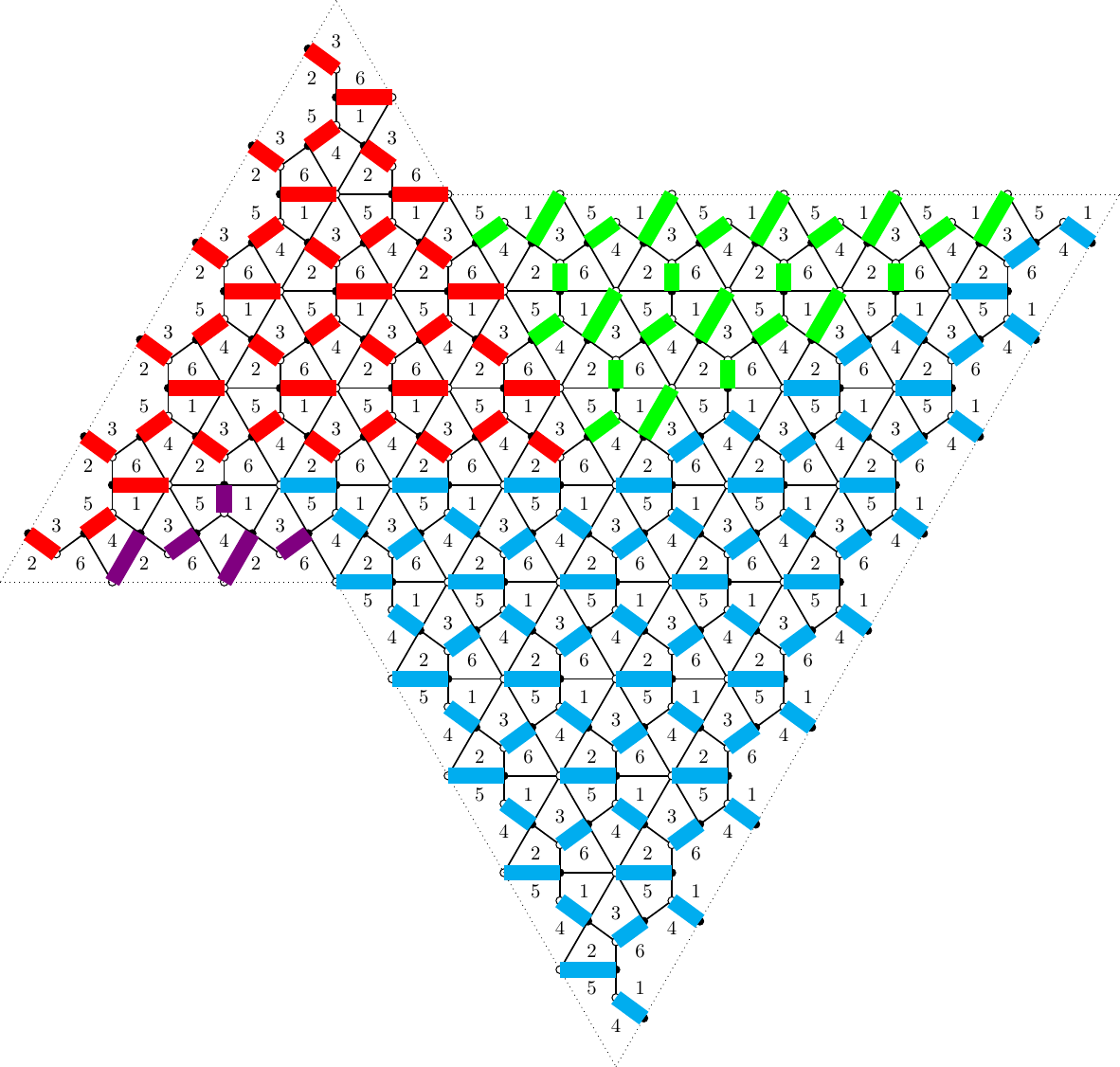}
            \caption{Minimal matching}
            \label{subfig:ConsMatching}
        \end{subfigure}
        \caption{Four sectors and minimal matching of Castle $\mathcal{C}_{4,3,2}$}
        \label{fig:ConsExample}
    \end{figure}
    Note that in Table~\ref{table:covering}, the matching when one side is positive is the same as when the next side is negative. This is because in step 1 when moving from a positive side to a negative side, we do nothing, so these two sides are incident to the same sector. Thus, they should have the same universal covering. 
    
    Also when two consecutive sides have the same sign, their universal covering has one edge in common. We will see that this leads to a "smooth transition" between the two corresponding sectors. As a result, there will be no twistable face along the straight line dividing the two sectors. The twistable faces do not come from the interior of each sector either. Therefore, they only come from the two staircases and the zero line.

    \begin{remark}
        There seems to be a small ambiguity in step 2 above for the covering on the zero line as we may have the choice of which covering to use. However, depending on the parity of the staircases' length, there is only a unique covering for this line.
    \end{remark}

    \begin{remark}\label{rem:zero-line-direct}
        In step 1 of the construction, one can see that the sign pattern of the contour determines the direction of the staircases, the straight lines, and the zero line. Recall from Remark \ref{rem:sign-pattern} that the sign patterns are $(+,-,\textcolor{red}{\Hat{+}},+,-,\textcolor{red}{\Hat{+}})$, $(\textcolor{red}{\Hat{-}},+,-,\textcolor{red}{\Hat{-}},+,-)$, $(\textcolor{red}{\Hat{+}},+,-,\textcolor{red}{\Hat{-}},+,-)$, $(-,\textcolor{red}{\Hat{-}},+,+,\textcolor{red}{\Hat{-}},+)$, $(\textcolor{red}{\Hat{+}},+,+,\textcolor{red}{\Hat{-}},+,-)$, and $(-,-,\textcolor{red}{\Hat{-}},+,-,\textcolor{red}{\Hat{+}})$. Here, the (opposing) sides with the red signs\footnote{and hats in case of reading this in grayscale} are those parallel to the zero line. For example, if the sign pattern is $(+,-,\textcolor{red}{\Hat{+}},+,-,\textcolor{red}{\Hat{+}})$, then the zero line is parallel to side $c$ and $f$. On the other hand, if the sign pattern is $(\textcolor{red}{\Hat{-}},+,-,\textcolor{red}{\Hat{-}},+,-)$, the zero line is parallel to side $a$ and $d$.
    \end{remark}

    \begin{thm}\label{thm:min-match}
        The construction above gives a valid perfect matching, and it is the minimal matching for Aztec Castles.
    \end{thm}

\subsection{The zero line}\label{subsec:zero-line}

    In this section, we prove the existence of the zero line.

\subsubsection{White regions}\label{subsubsec:white-zero}

    Recall from Remark \ref{rem:sign-change} that in the white regions, the sign patterns are either $(+,-,\textcolor{red}{+},+,-,\textcolor{red}{+})$, $(\textcolor{red}{-},+,-,\textcolor{red}{-},+,-)$, or their cyclic rotations. Recall also that the sides with the red signs are those parallel to the zero line. For our convenience, we call the top right white region in Figure \ref{fig:allRegionPos} \textbf{Region 1}, where $k\geq 1; k - 1 \leq i,j$. Likewise, we call the top  middle region in Figure \ref{fig:allRegionPos} \textbf{Region 1'}, where $k\geq 1; k - 1 \leq i+j; i\leq -k$. Note that by Remark \ref{rem:sign-change}, discussion of these 2 cases is sufficient.
    
    \textbf{Region 1}. This region has sign pattern $(+,-,\textcolor{red}{+},+,-,\textcolor{red}{+})$, so the zero line is parallel to sides $c$ and $f$. Figure \ref{fig:zeroLineR1} shows the calculations for this region. Side $a$ has length $j+k$ while side $b$ has length $i+j+k$, so the difference between sides $f$ and $c$ is $|i| = i$. Side $c$ has length $i+k$, so the height of the staircase is $\left\lfloor\dfrac{|i+k|}{2}\right\rfloor = \left\lfloor\dfrac{i+k}{2}\right\rfloor$. Side $f$ has length $i+1-k$, so the height of the staircase is $\left\lfloor\dfrac{|i+1-k|}{2}\right\rfloor = \left\lfloor\dfrac{i+1-k}{2}\right\rfloor$. It remains to show that $\left\lfloor\dfrac{i+k}{2}\right\rfloor + \left\lfloor\dfrac{i+1-k}{2}\right\rfloor = i$, which is straightforward. If $i - k = 2\ell$, then
    \[ \left\lfloor\dfrac{i+k}{2}\right\rfloor + \left\lfloor\dfrac{i+1-k}{2}\right\rfloor = k+\ell + \ell = k + 2\ell = i. \]
    If $i - k = 2\ell + 1$, then
    \[ \left\lfloor\dfrac{i+k}{2}\right\rfloor + \left\lfloor\dfrac{i+1-k}{2}\right\rfloor = k+\ell + \ell + 1 = k + 2\ell + 1 = i. \]
    The arguments for other shapes whose sign patterns are cyclic rotations of $(+,-,+,+,-,+)$ are analogous, and they come down to checking one of the following three identities:

    \begin{align*}
        \left\lfloor\dfrac{i+k}{2}\right\rfloor + \left\lfloor\dfrac{i+1-k}{2}\right\rfloor &= i, \\ 
        \left\lfloor\dfrac{j+k}{2}\right\rfloor + \left\lfloor\dfrac{j+1-k}{2}\right\rfloor &= j, \\ 
        \left\lfloor\dfrac{-i-j-k}{2}\right\rfloor + \left\lfloor\dfrac{-i-j-1+k}{2}\right\rfloor &= -i-j-1.
    \end{align*}

    \begin{figure}[h!]
        \centering
        \begin{minipage}{.5\textwidth}
            \centering
            \includegraphics[scale = 0.25]{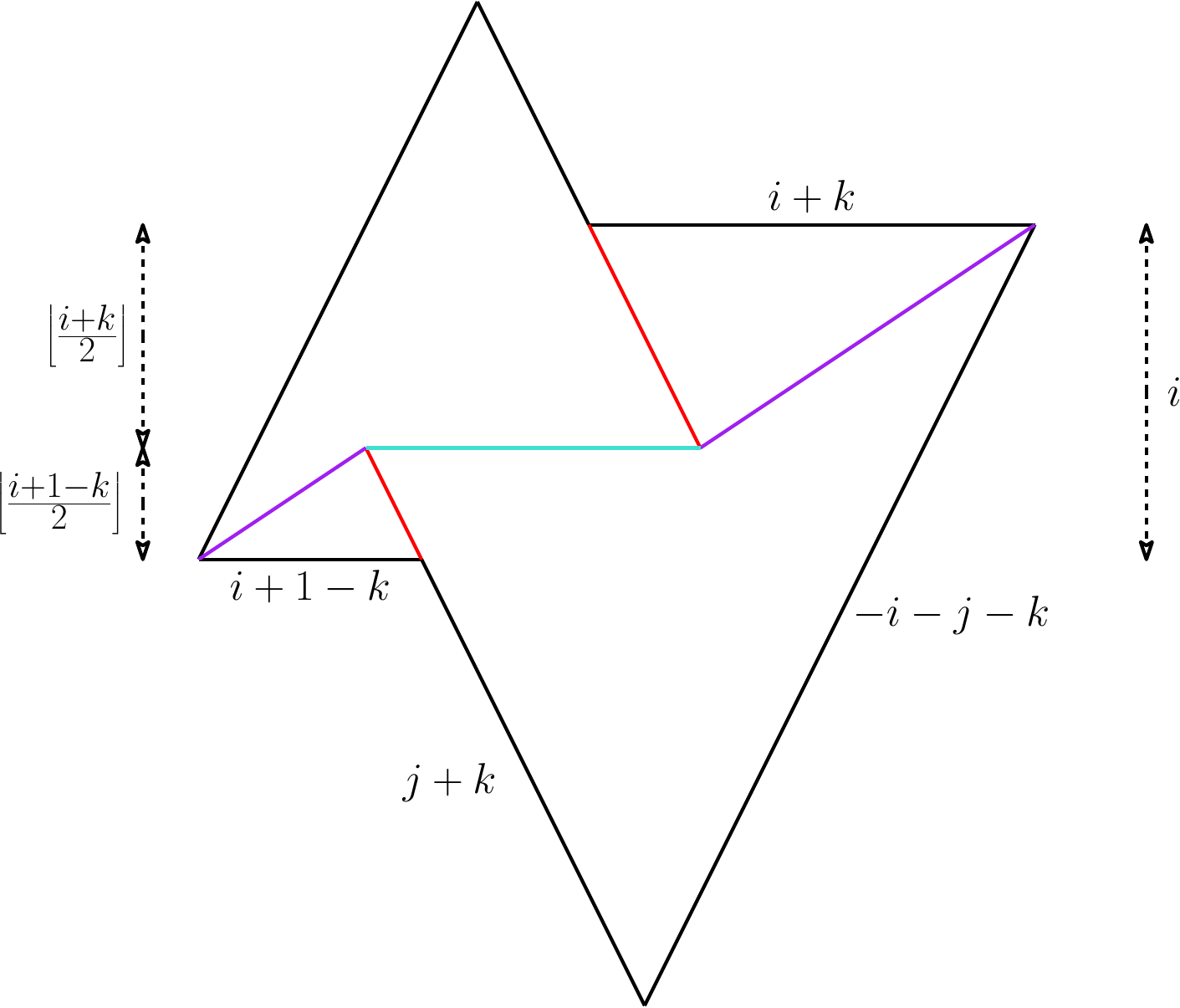}
            \caption{Region 1 zero line}
            \label{fig:zeroLineR1}
        \end{minipage}%
        \begin{minipage}{.5\textwidth}
            \centering
            \includegraphics[scale = 0.25]{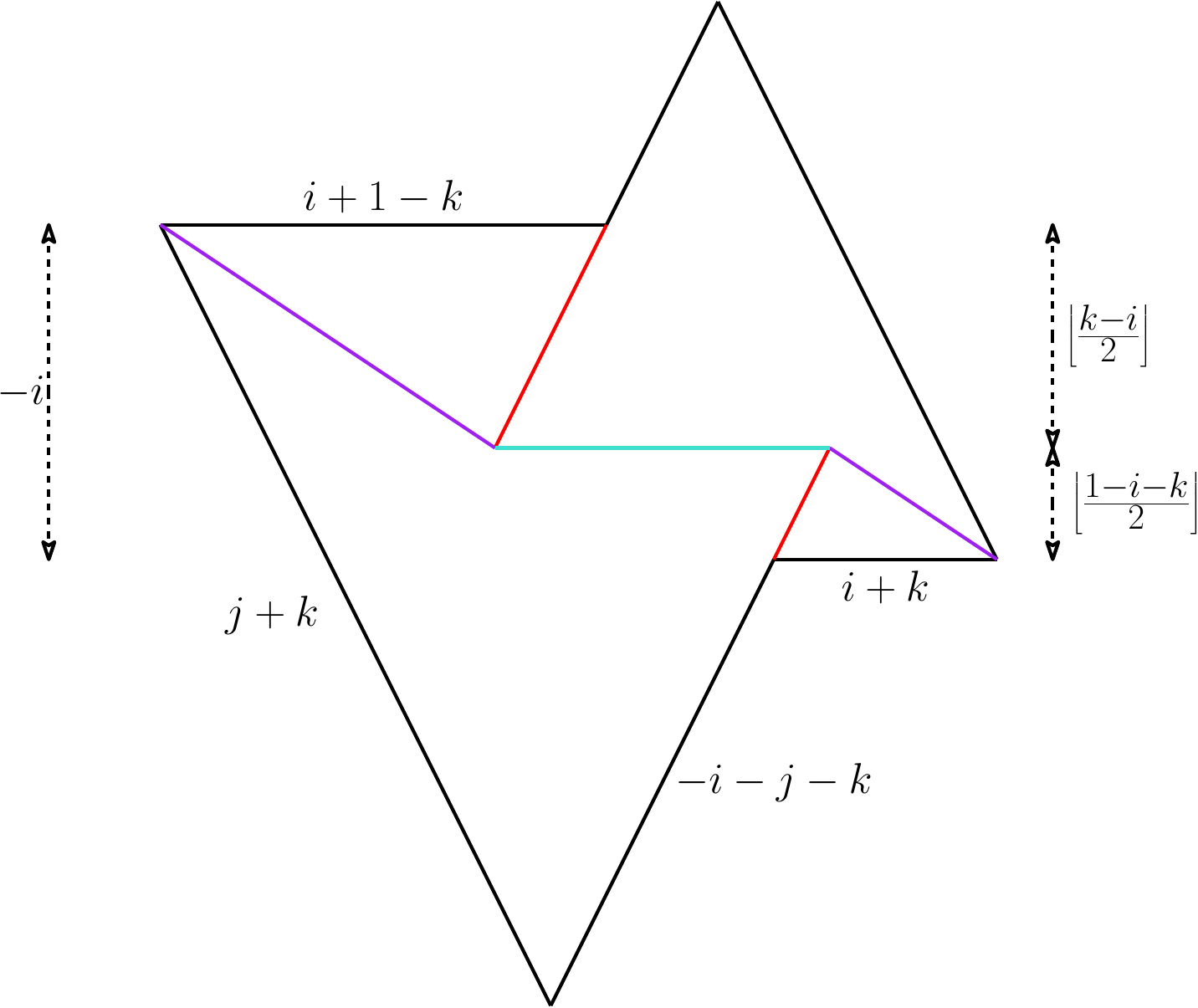}
        \caption{Region 1' zero line}
        \label{fig:zeroLineR1a}
        \end{minipage}
    \end{figure}

    \textbf{Region 1'}. This region has sign pattern $(+,-,\textcolor{red}{-},+,-,\textcolor{red}{-})$, so the zero line is parallel to sides $c$ and $f$. Figure \ref{fig:zeroLineR1a} shows the calculations for this region. The difference between sides $f$ and $c$ is $|i| = -i$. Side $c$ has length $|i+k|$, but the first step of the staircase in on side $d$, so the height of the staircase is $\left\lfloor\dfrac{|i+k|+1}{2}\right\rfloor = \left\lfloor\dfrac{1-i-k}{2}\right\rfloor$. Side $f$ has length $|i+1-k|$, so the height of the staircase is $\left\lfloor\dfrac{|i+1-k|+1}{2}\right\rfloor = \left\lfloor\dfrac{k-i}{2}\right\rfloor$. It remains to show that $\left\lfloor\dfrac{1-i-k}{2}\right\rfloor + \left\lfloor\dfrac{k-i}{2}\right\rfloor = -i$, which is also straightforward.

    The arguments for other shapes whose sign patterns are cyclic rotations of $(-,+,-,-,+,-)$ are analogous, and they come down to checking one of the following three identities:

    \begin{align*}
        \left\lfloor\dfrac{1-i-k}{2}\right\rfloor + \left\lfloor\dfrac{k-i}{2}\right\rfloor &= -i, \\ 
        \left\lfloor\dfrac{1-j-k}{2}\right\rfloor + \left\lfloor\dfrac{k-j}{2}\right\rfloor &= -j, \\ 
        \left\lfloor\dfrac{i+j+k+1}{2}\right\rfloor + \left\lfloor\dfrac{i+j-k+2}{2}\right\rfloor &= i+j+1.
    \end{align*}

\subsubsection{Pink regions}\label{subsubsec:pink-zero}

    In the pink regions, the sign patterns are $(\textcolor{red}{+},+,-,\textcolor{red}{-},+,-)$, $(-,\textcolor{red}{-},+,+,\textcolor{red}{-},+)$, and their cyclic rotations. We call the top middle pink region in Figure \ref{fig:allRegionPos} \textbf{Region 2}, where $k\geq 1; -k \leq i \leq k-1 \leq j, i+j $. Also, we call the top left pink region in Figure \ref{fig:allRegionPos} \textbf{Region 2'}, where $k\geq 1; i\leq -k \leq i+j\leq k-1\leq j$. Again, by Remark \ref{rem:sign-change}, discussion of these 2 cases is sufficient. 
    
    \textbf{Region 2}. This region has sign pattern $(+,-,\textcolor{red}{+},+,-,\textcolor{red}{-})$, so the zero line is parallel to sides $c$ and $f$. Figure \ref{fig:zeroLineR2} shows the calculations for this region. The calculations are almost the same as in region 1. The height of the staircases in this case are $\left\lfloor\dfrac{|i+k|}{2}\right\rfloor = \left\lfloor\dfrac{i+k}{2}\right\rfloor$ and $\left\lfloor\dfrac{|k-1-i|+1}{2}\right\rfloor = \left\lfloor\dfrac{k-i}{2}\right\rfloor$. We need to check $\left\lfloor\dfrac{i+k}{2}\right\rfloor - \left\lfloor\dfrac{k-i}{2}\right\rfloor = i$. This is also straightforward.

    Similar to the white regions, the arguments for other shapes whose sign patterns are cyclic rotations of $(+,+,-,-,+,-)$ comes down to checking analogous identities. This is left to the readers.

    \begin{figure}[h!]
        \centering
        \begin{minipage}[b]{.45\textwidth}
            \centering
            \includegraphics[scale = 0.25]{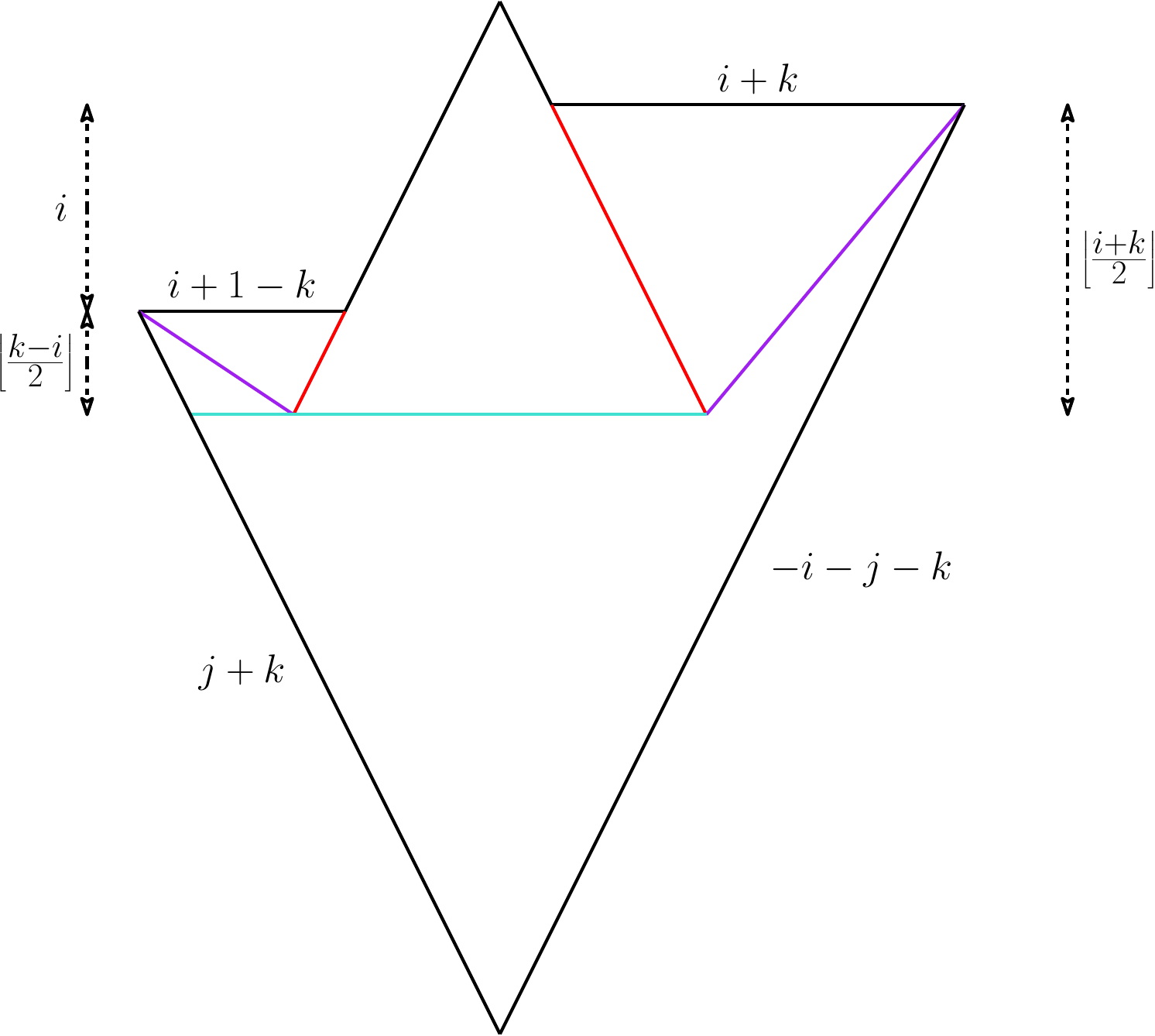}
            \caption{Region 2 zero line}
            \label{fig:zeroLineR2}
        \end{minipage}\quad\quad
        \begin{minipage}[b]{.45\textwidth}
            \centering
            \includegraphics[scale = 0.25]{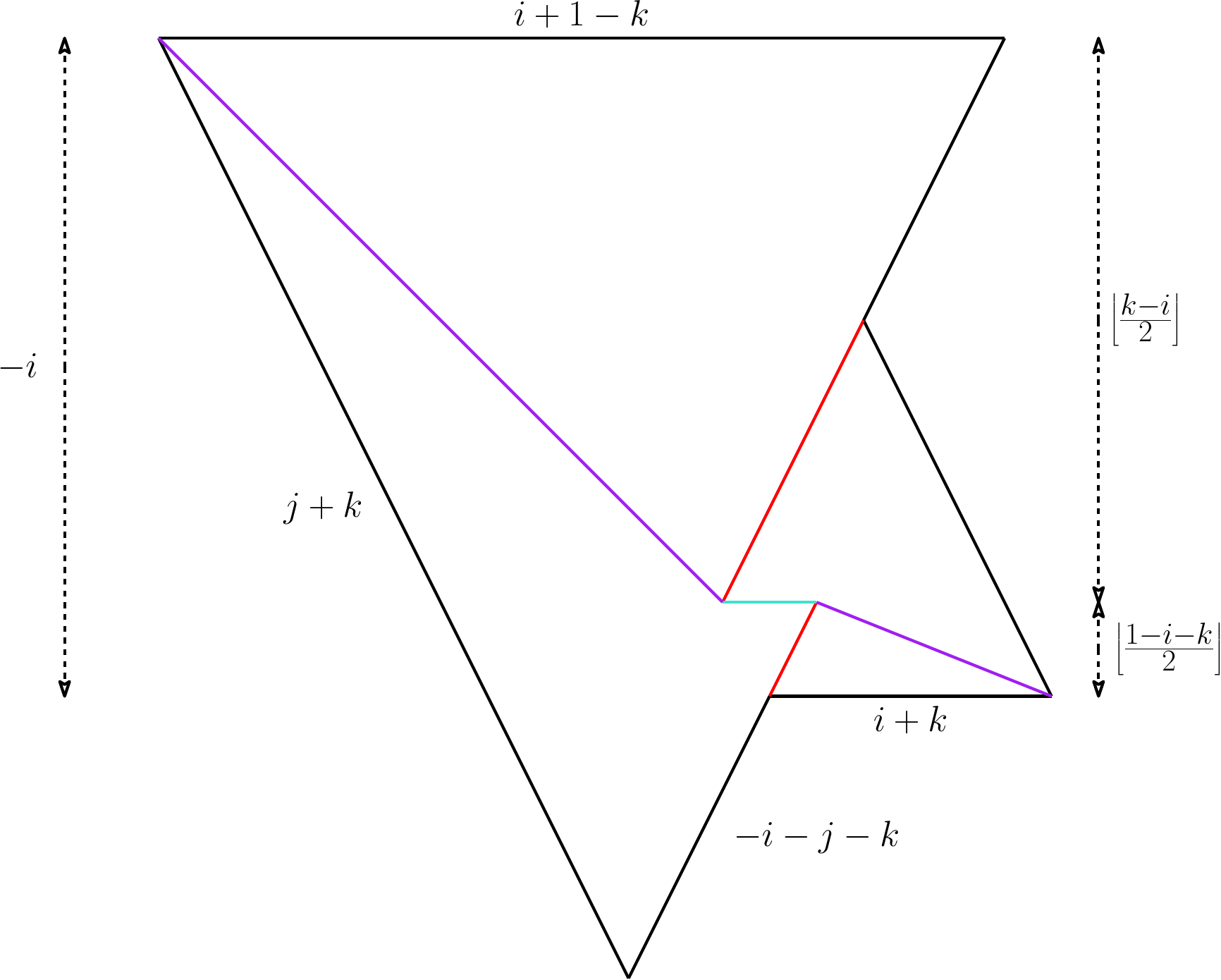}
        \caption{Region 2' zero line}
        \label{fig:zeroLineR2a}
        \end{minipage}
    \end{figure}

    \textbf{Region 2'}. This region has sign pattern $(+,-,\textcolor{red}{-},+,+,\textcolor{red}{-})$, so the zero line is parallel to sides $c$ and $f$. Figure \ref{fig:zeroLineR2a} shows the calculations for this region. Observe that the calculations for this region are exactly the same as for region 1' above. Hence, the argument for this case is the same as in Section \ref{subsubsec:white-zero}.

\subsubsection{Blue regions}\label{subsubsec:blue-zero}

    Lastly, in the blue regions, the sign patterns are $(\textcolor{red}{+},+,+,\textcolor{red}{-},+,-)$, and $(-,-,\textcolor{red}{-},+,-,\textcolor{red}{+})$, and their cyclic rotations. We call the blue region in the first quadrant in Figure \ref{fig:allRegionPos} \textbf{Region 3}, where $k\geq 1; i,j \leq k-1 \leq i+j$. Likewise, we call the top blue region in the second quadrant in Figure \ref{fig:allRegionPos} \textbf{Region 3'}, where $k\geq 1; i+j\leq k-1 \leq j; i\geq -k$. It suffices to discuss these 2 cases by Remark \ref{rem:sign-change}.
    
    \textbf{Region 3}. This region has sign pattern $(+,-,\textcolor{red}{+},+,+,\textcolor{red}{-})$, so the zero line is parallel to sides $c$ and $f$. Figure \ref{fig:zeroLineR3} shows the calculations for this region. Observe that the calculations for this region are exactly the same as for region 2 above. Hence, the argument is the same as in Section \ref{subsubsec:pink-zero}.

    \begin{figure}[h!]
        \centering
        \begin{minipage}[b]{.45\textwidth}
            \centering
            \includegraphics[scale = 0.25]{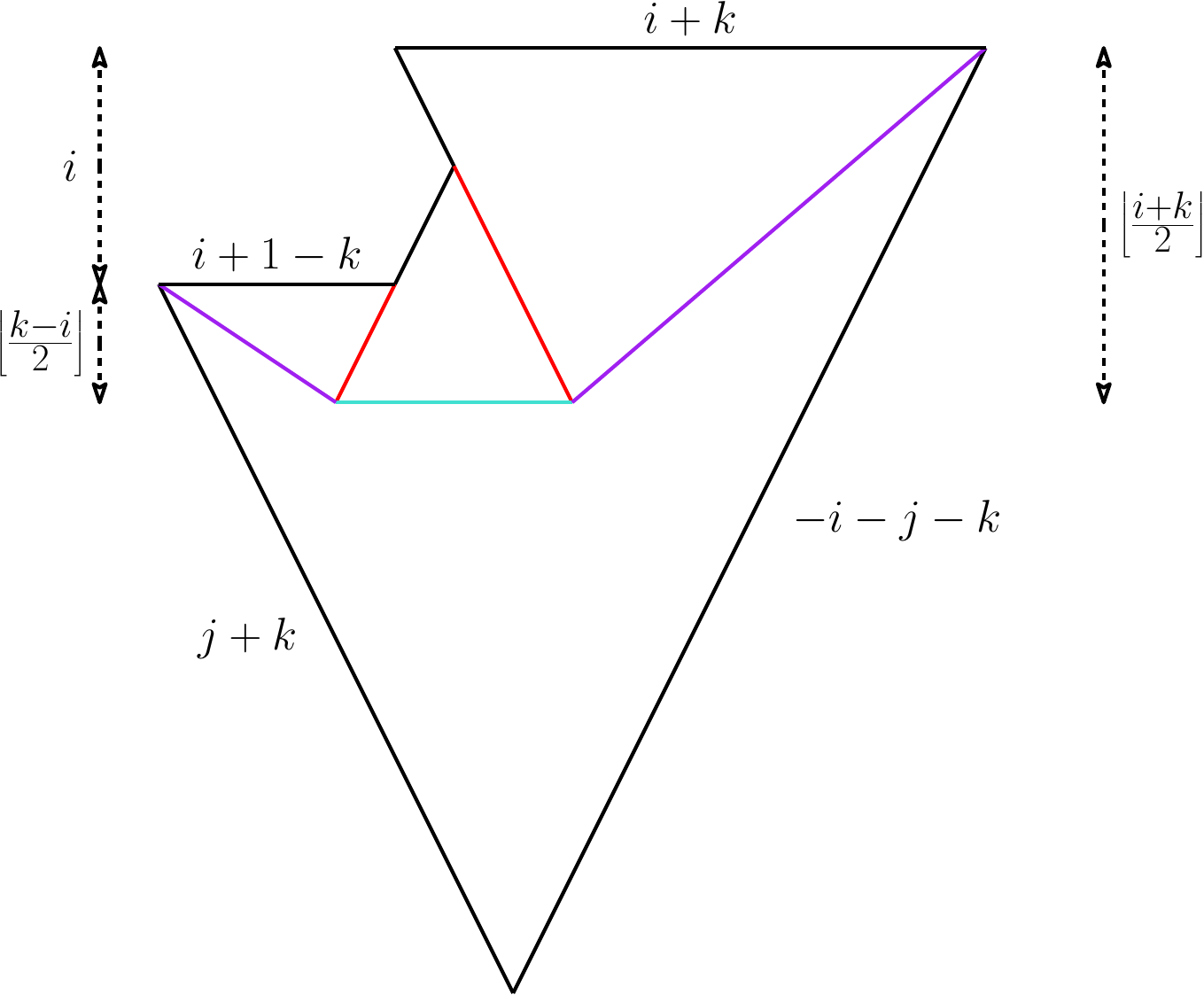}
            \caption{Region 3 zero line}
            \label{fig:zeroLineR3}
        \end{minipage}\quad\quad
        \begin{minipage}[b]{.45\textwidth}
            \centering
            \includegraphics[scale = 0.25]{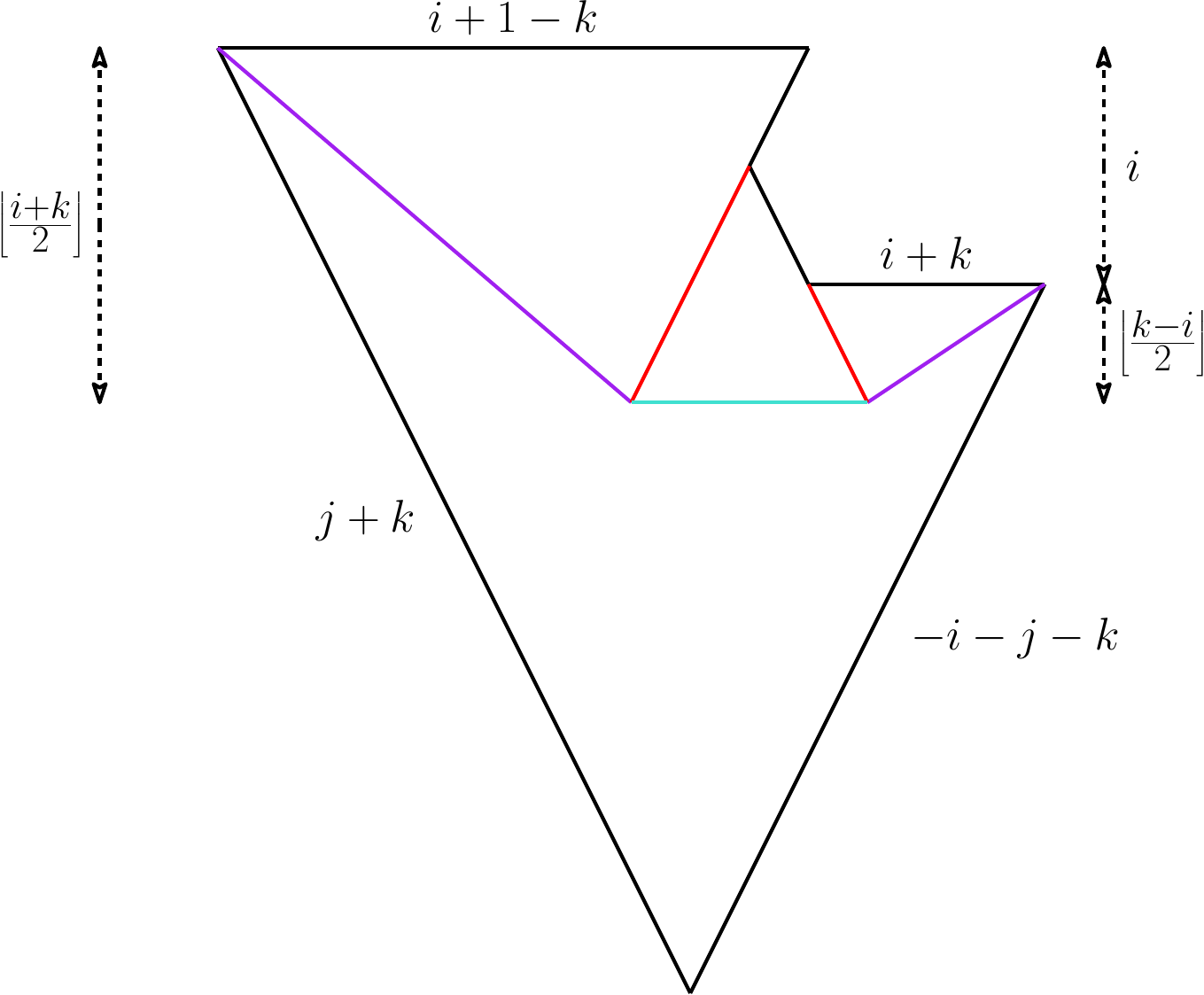}
        \caption{Region 3' zero line}
        \label{fig:zeroLineR3a}
        \end{minipage}
    \end{figure}

    \textbf{Region 3'}. This region has sign pattern $(+,-,\textcolor{red}{+},-,-,\textcolor{red}{-})$, so the zero line is parallel to sides $c$ and $f$. Figure \ref{fig:zeroLineR3a} shows the calculations for this region. Observe that the calculations for this region are exactly the same as for region 3 above. Hence, the argument is also the same as in Section \ref{subsubsec:pink-zero}.

\section{Proof of main theorem}\label{sec:proof}

\subsection{Proof sketch}\label{subsec:proof-sketch}

    In this section, we give the proof of Theorem \ref{thm:min-match}. The proof consists of the following points.

    \begin{enumerate}
        \item There is no twistable face in the interior of each sector.
        \item There is no twistable face along the straight-line borders.
        \item All twistable faces along the staircase borders are negative.
        \item All twistable faces along the zero line are negative.
    \end{enumerate}

    Along the way, we will also show that the proposed construction gives a valid perfect matching in which every vertex is covered exactly once. Then, Theorem \ref{thm:min-match} follows by Corollary \ref{cor:twist-condition}.

\subsection{Proof}\label{subsec:proof}

    \begin{lemma}\label{lem:interior}
        In the interior of each sector, there is no twistable face, and every vertex is covered exactly once.
    \end{lemma}

    \begin{proof}
        This lemma is immediate from examining the universal coverings in Table \ref{table:covering}.  In particular, every face is incident to exactly one such edge of the matching.
    \end{proof}

    \begin{lemma}\label{lem:straight}
        Along the straight line borders, there is no twistable face, and every vertex is covered exactly once.
    \end{lemma}

    \begin{proof}
        Recall from the construction that there are two scenarios that create a straight line border:
        \begin{enumerate}
            \item we move from a positive side to a positive side, and draw a straight line in the direction of the second side; or
            \item we move from a negative side to a negative side, and draw a straight line in the direction of the first side.
        \end{enumerate}
        Let us show the argument for the first case. The second case follows analogously. In the first case, the matching along the straight line border is as follows:

        \begin{itemize}
            \item The black vertices are covered by the covering in the first side's sector.
            \item The white vertices are covered by the covering in the second side's sector.
            \item No edge along the border is covered, and hence there is no twistable face.
        \end{itemize}
        
        Figure \ref{fig:posAposB} shows an example when we move from a positive side $a$ to a positive side $b$, the straight line is in the direction of side $b$. The black vertices on the border are covered by the edges $1-4$ in side $a$'s sector. The white vertices on the border are covered by the edges $3-5$ in side $b$'s sector. There is no edges covered along the border, and there is no twistable face. We encourage the readers to check this observation for other pairs of consecutive positive sides.

        \begin{figure}[h!]
            \centering
            \includegraphics[scale = 0.7]{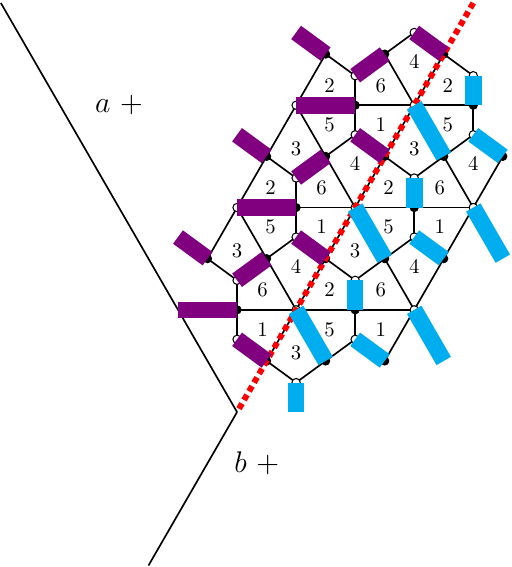}
            \caption{Straight line border (colored \textcolor{red}{red}) between positive $a$ and positive $b$}
            \label{fig:posAposB}
        \end{figure}
    \end{proof}
    \begin{lemma}\label{lem:staircase}
        Along the staircase borders, all twistable faces are negative, and every vertex is covered exactly once.
    \end{lemma}

    \begin{proof}
        Recall from the construction that we draw a staircase border we move from a negative side to a positive side, with the first step lying on the positive side. We call a step on the staircase a \textit{base step} (resp. \textit{side step}) if it is parallel to the positive (resp. negative) side. There are two white vertices and one black vertex on each side step. We call a white vertex on a side step a \textit{white concave-up vertex} (colored red) if the arrow from it to the black vertex on the same step points in the negative direction of the negative side. We call the other white vertex the \textit{white concave-down vertex} (colored blue). 
        The matching along the staircase border is as follows:

        \begin{itemize}
            \item On the base steps, the black vertices are covered by the covering in the positive side's sector.
            \item On the side steps, the white concave-up vertices are covered by the covering in the negative side's sector; and
            \item the black vertices and the white concave-down vertices are matched following the positive side's sector. The faces containing these edges on the negative side's sector are the negative twistable faces.
        \end{itemize}

        Figure \ref{fig:negBposC} shows an example when we move from a negative side $b$ to a positive side $c$. The black vertices on the base steps are covered by the edges $2-6$ in side $c$'s sector. The white concave-up vertices on the side steps are covered by the edges $2-5$ in side $b$'s sector. The remaining vertices on the side steps are matched by the edges $1-3$. The faces $3$ containing these edges on side $b$'s sector are the negative twistable faces. We encourage the readers to check this observation for other pairs of consecutive positive sides.

    \begin{figure}[h!]
    \centering
        \begin{subfigure}[b]{0.45\textwidth}
            \centering
            \includegraphics[width = \textwidth]{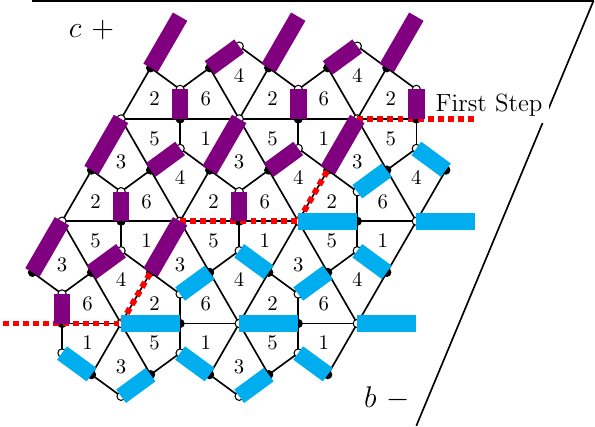}
            \caption{The matching}
            \label{fig:negBposC}
        \end{subfigure}
     \quad
        \begin{subfigure}[b]{0.45\textwidth}
            \centering
            \includegraphics[width = \textwidth]{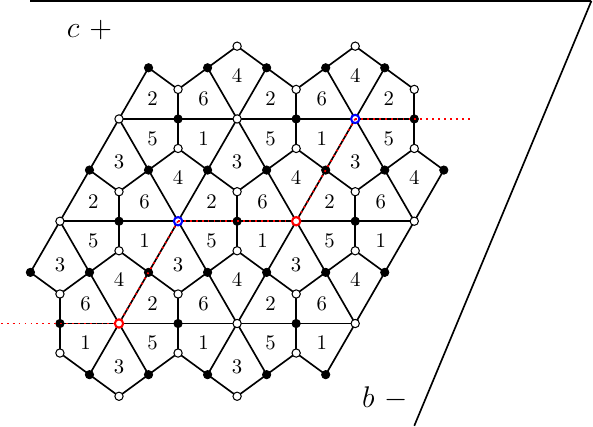}
         \caption{White concave up and down vertices}
            \label{subfig:whitevertex}
        \end{subfigure}
        \caption{Staircase border between $b\leq 0$ and $c\geq 0$}
        \label{fig:StaircaseExample}
    \end{figure}
    \end{proof}

    \begin{lemma}\label{lem:zeroLine}
        Along the zero line, all twistable faces are negative, and every vertex is covered exactly once.
    \end{lemma}

    \begin{proof}
        Similar to Section \ref{subsec:zero-line}, we only need to show the proof for three representative regions:
        \begin{enumerate}
            \item $k\geq 1; k - 1 \leq i,j$;
            \item $k\geq 1; -k \leq i \leq k-1 \leq j, i+j$; and
            \item $k\geq 1; i,j \leq k-1 \leq i+j$.
        \end{enumerate}
        Note that in all three regions, the two sectors incident to the zero line are incident to a negative side $b$ and a negative side $e$.
        
        In region 1, as shown in Section \ref{subsec:zero-line}, the two staircases have lengths of different parity and base steps on parallel sides ($c$ and $f$), so one of them ends with a side step and the other end with a base step. 

        By the matching rule along the staircase borders in Lemma \ref{lem:staircase}, of the two endpoints of the zero line, only one is covered. Thus, there are an even number of uncovered vertices on the zero line. If the covered endpoint is covered by the covering of side $b$ (resp. side $e$), then we can cover the rest of the zero line by the same covering as in Figure \ref{subfig:negBnegEcase1} (resp. Figure \ref{subfig:negBnegEcase2}). We refer the readers to Figure \ref{fig:432Example} and \ref{fig:532Example} for examples of these coverings. In both cases, the twistable faces containing these edges on the zero line are negative.

        \begin{figure}[h!]
        \centering
            \begin{subfigure}[b]{0.4\textwidth}
                \centering
                \includegraphics[width = \textwidth]{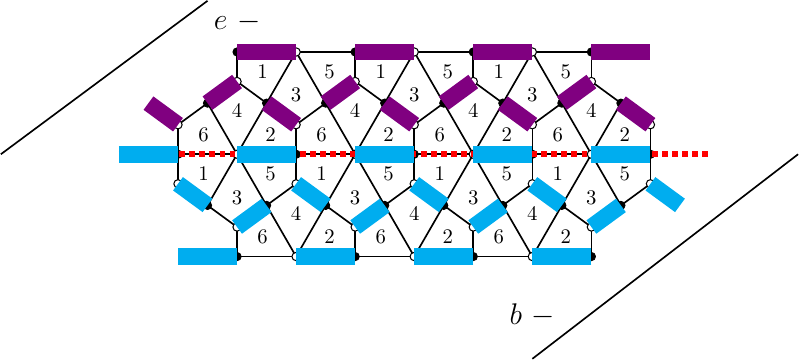}
                \caption{Zero line's matching follows negative $b$}
                \label{subfig:negBnegEcase1}
            \end{subfigure}
         \quad
            \begin{subfigure}[b]{0.4\textwidth}
                \centering
                \includegraphics[width = \textwidth]{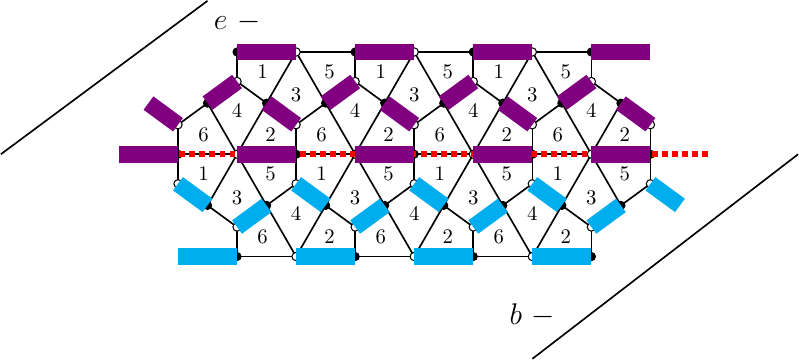}
                \caption{Zero line's matching follows negative $e$}
                \label{subfig:negBnegEcase2}
            \end{subfigure}
            \caption{Zero line border between negative $b$ and negative $e$}
            \label{fig:negBnegE}
        \end{figure}

        In regions 2 and 3, the two staircases have lengths of the same parity but base steps on non-parallel sides ($c$ and $a$), so the same is true: one of them ends with a side step and the other ends with a base step. Thus, the arguments follow analogously. We refer the readers to Figure \ref{fig:043Example} and \ref{fig:044Example} for examples of the coverings.
    \end{proof}

\section*{Acknowledgements}
    
    This project was partially supported by the RTG grant NSF/DMS-1745638 and the FRG grant NSF/DMS-1854162. It was supervised as part of the University of Minnesota School of Mathematics Summer 2022 REU program. The authors would like to thank Carolyn Stephen and Sylvester Zhang for their continuous support throughout the project. We would also like to thank David Speyer for helpful discussions.

\bibliography{bibliography}
\bibliographystyle{alpha}

\newpage
	
\appendix

\section{Minimal matching examples}\label{append:examples}

    Here we show examples of minimal matchings of Aztec Castles of different shapes, corresponding to different regions of $\mathbb{Z}^3$. Here the red lines correspond to the straight lines and staircases, and the blue lines correspond to the zero lines.

    \textbf{Region 1:} $k\geq 1; k - 1 < i,j$. This is the top right white region in Figure \ref{fig:allRegionPos}.

    \begin{figure}[h!]
    \centering
        \begin{subfigure}[b]{0.4\textwidth}
            \centering
            \includegraphics[width = \textwidth]{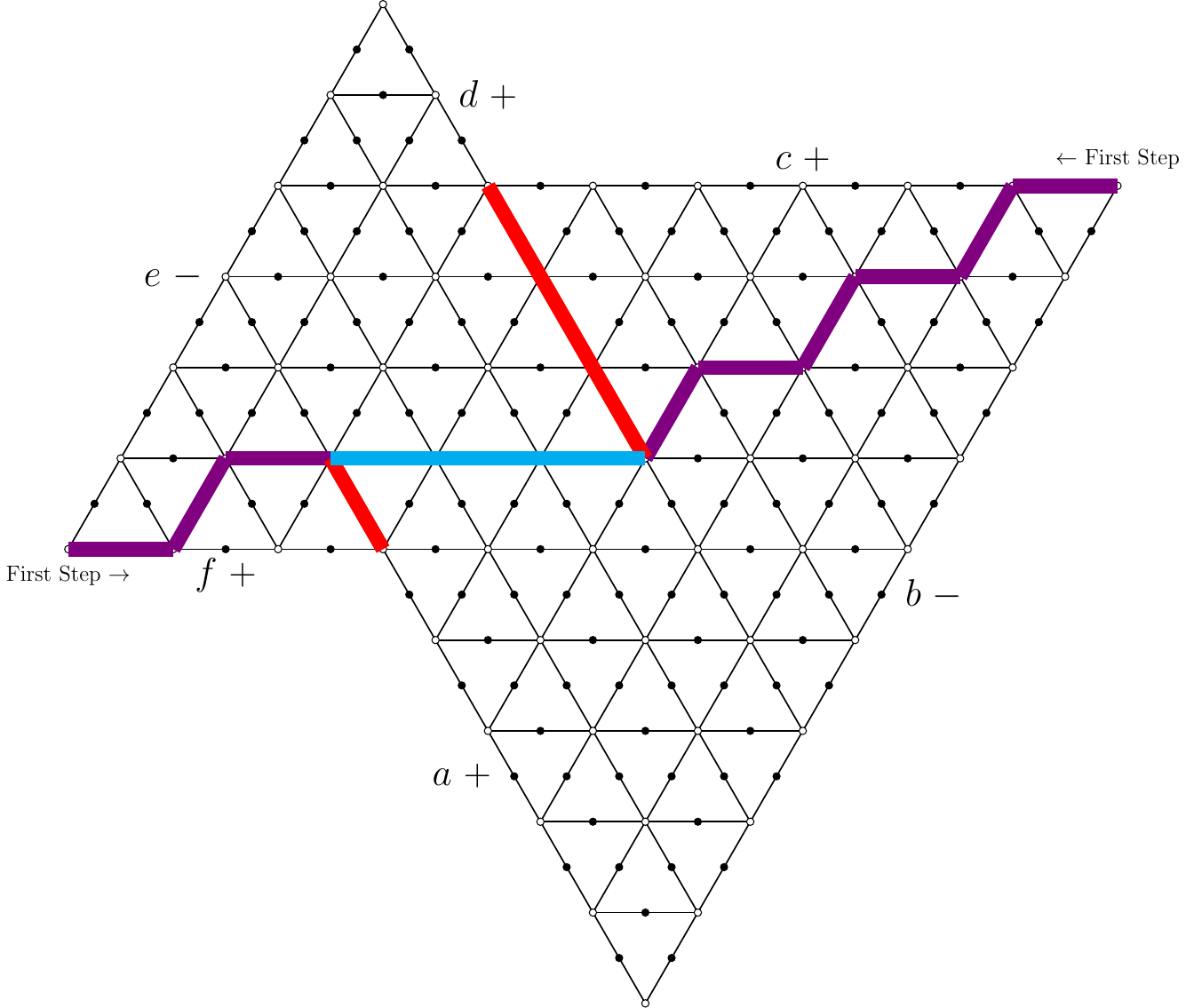}
            \caption{Four sectors}
            \label{subfig:432Divide}
        \end{subfigure}
     \quad
        \begin{subfigure}[b]{0.4\textwidth}
            \centering
            \includegraphics[width = \textwidth]{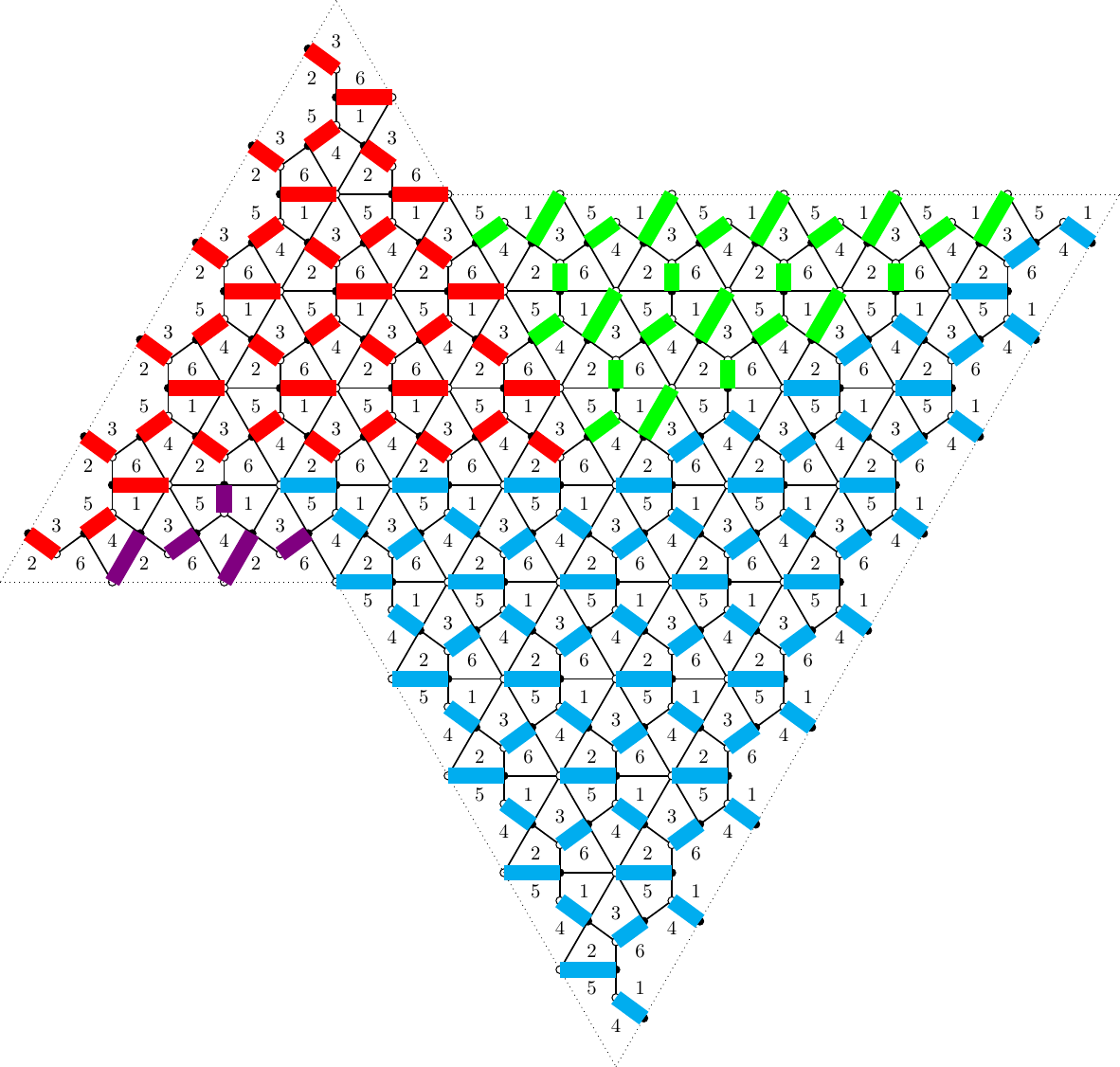}
            \caption{Minimal matching}
            \label{subfig:432Matching}
        \end{subfigure}
        \caption{Four sectors and minimal matching of Castle $\mathcal{C}_{4,3,2}$ }
        \label{fig:432Example}
    \end{figure}

    \begin{figure}[h!]
    \centering
        \begin{subfigure}[b]{0.4\textwidth}
            \centering
            \includegraphics[width = \textwidth]{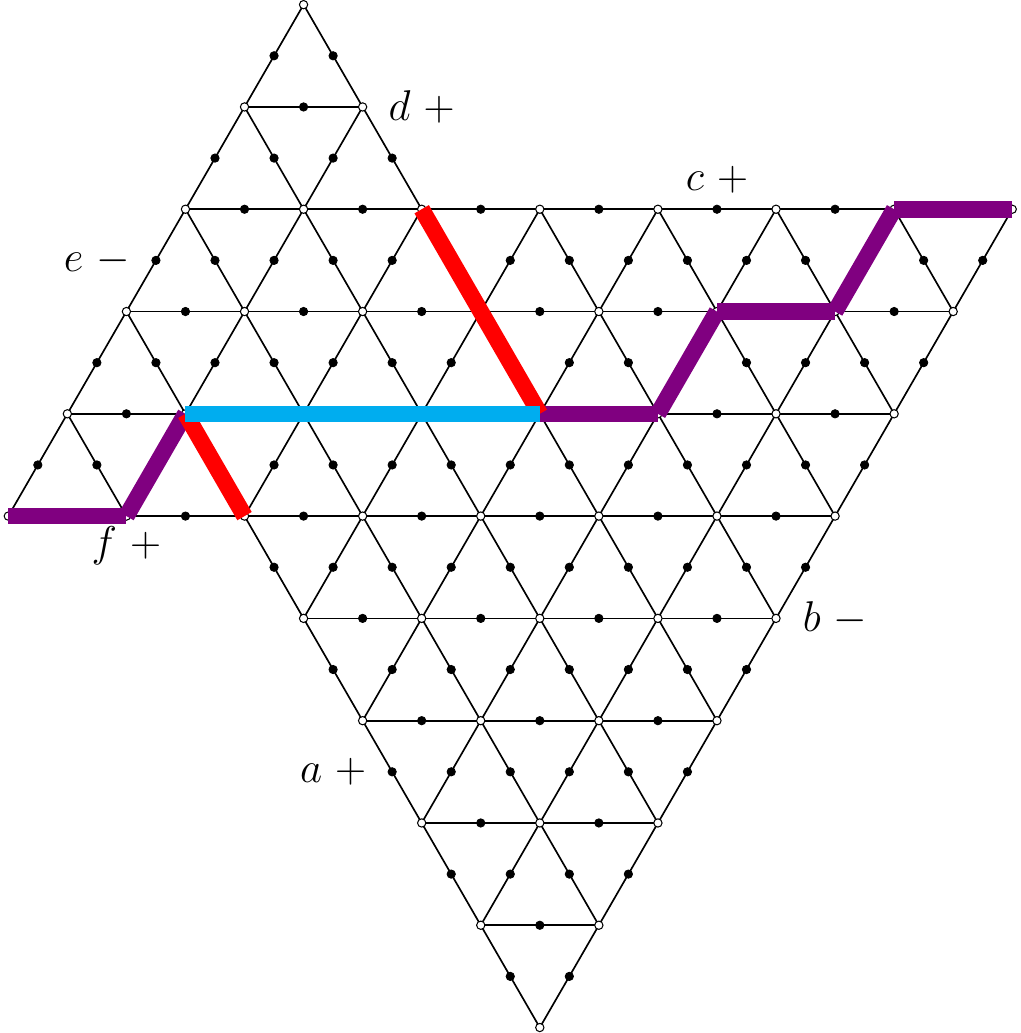}
            \caption{Four sectors}
            \label{subfig:532Divide}
        \end{subfigure}
     \quad
        \begin{subfigure}[b]{0.4\textwidth}
            \centering
            \includegraphics[width = \textwidth]{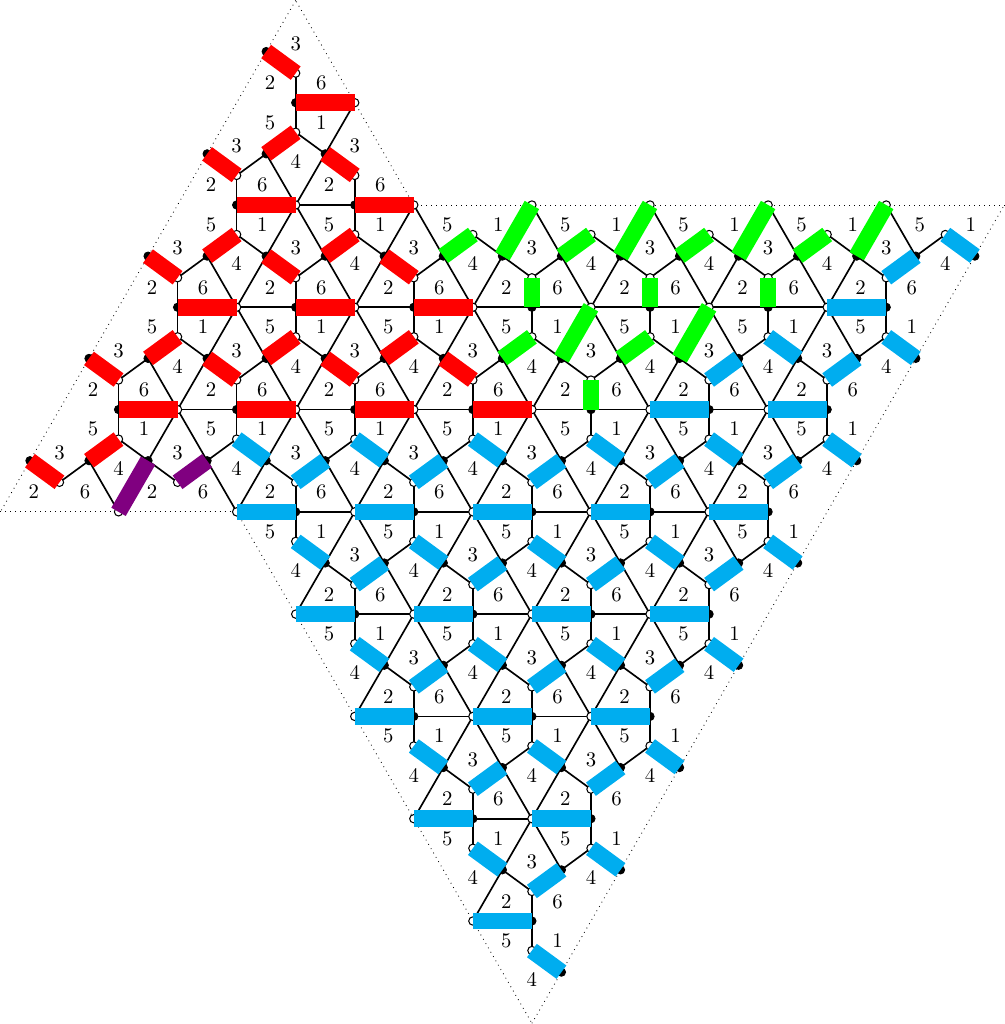}
            \caption{Minimal matching}
            \label{subfig:532Matching}
        \end{subfigure}
        \caption{Four sectors and minimal matching of Castle $\mathcal{C}_{5,3,2}$ }
        \label{fig:532Example}
    \end{figure}
    
\newpage
    \textbf{Region 2:} $k\geq 1; -k < i < k-1 < j, i+j $. This is the top middle pink region in Figure \ref{fig:allRegionPos}.

    \begin{figure}[h!]
    \centering
        \begin{subfigure}[b]{0.4\textwidth}
            \centering
            \includegraphics[width = \textwidth]{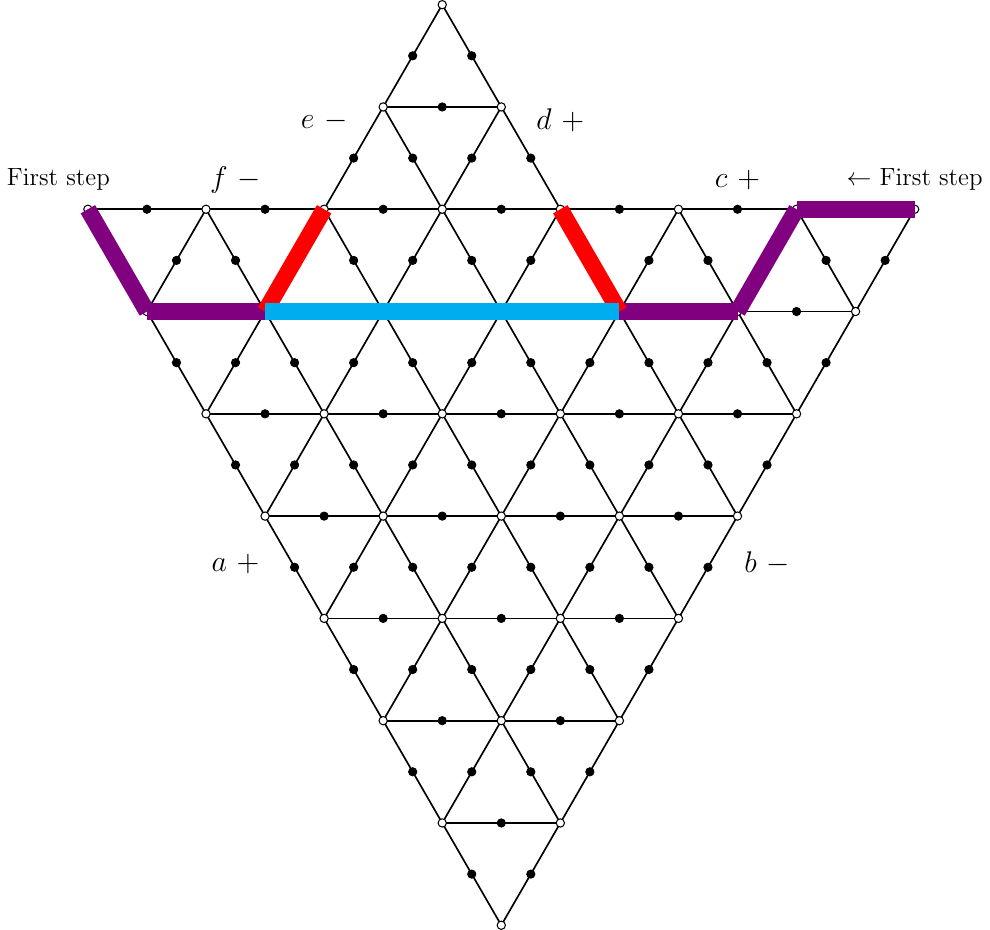}
            \caption{Four sectors}
            \label{subfig:043Divide}
        \end{subfigure}
     \quad
        \begin{subfigure}[b]{0.4\textwidth}
            \centering
            \includegraphics[width = \textwidth]{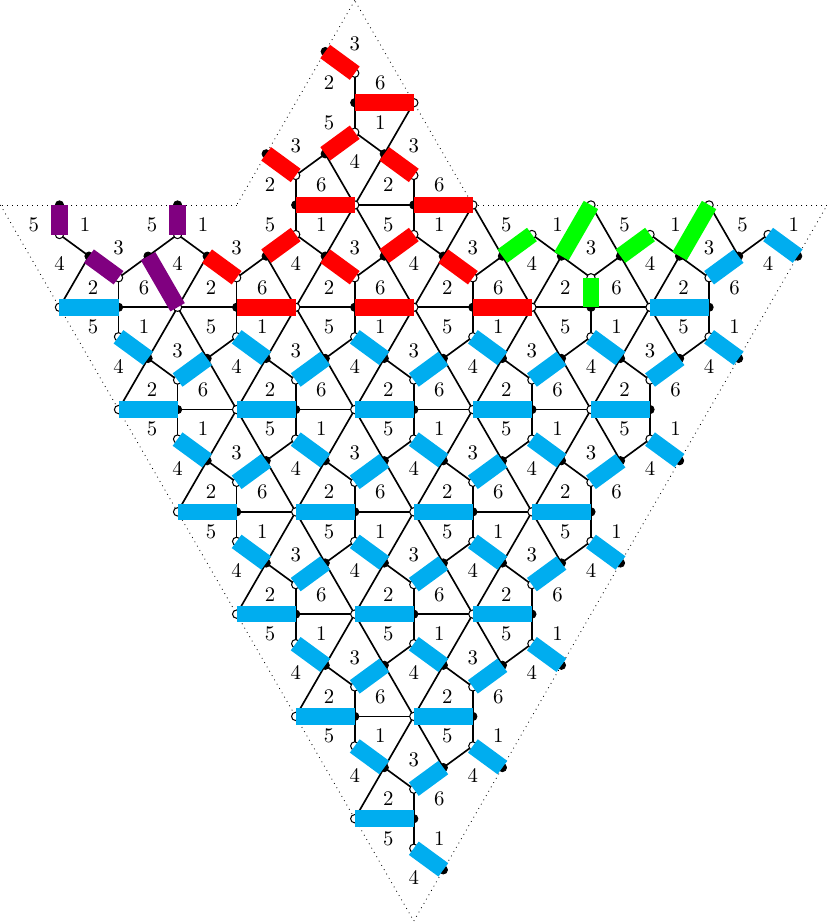}
            \caption{Minimal matching}
            \label{subfig:043Matching}
        \end{subfigure}
        \caption{Four sectors and minimal matching of Castle $\mathcal{C}_{0,4,3}$}
        \label{fig:043Example}
    \end{figure}

    \begin{figure}[h!]
    \centering
        \begin{subfigure}[b]{0.4\textwidth}
            \centering
            \includegraphics[width = \textwidth]{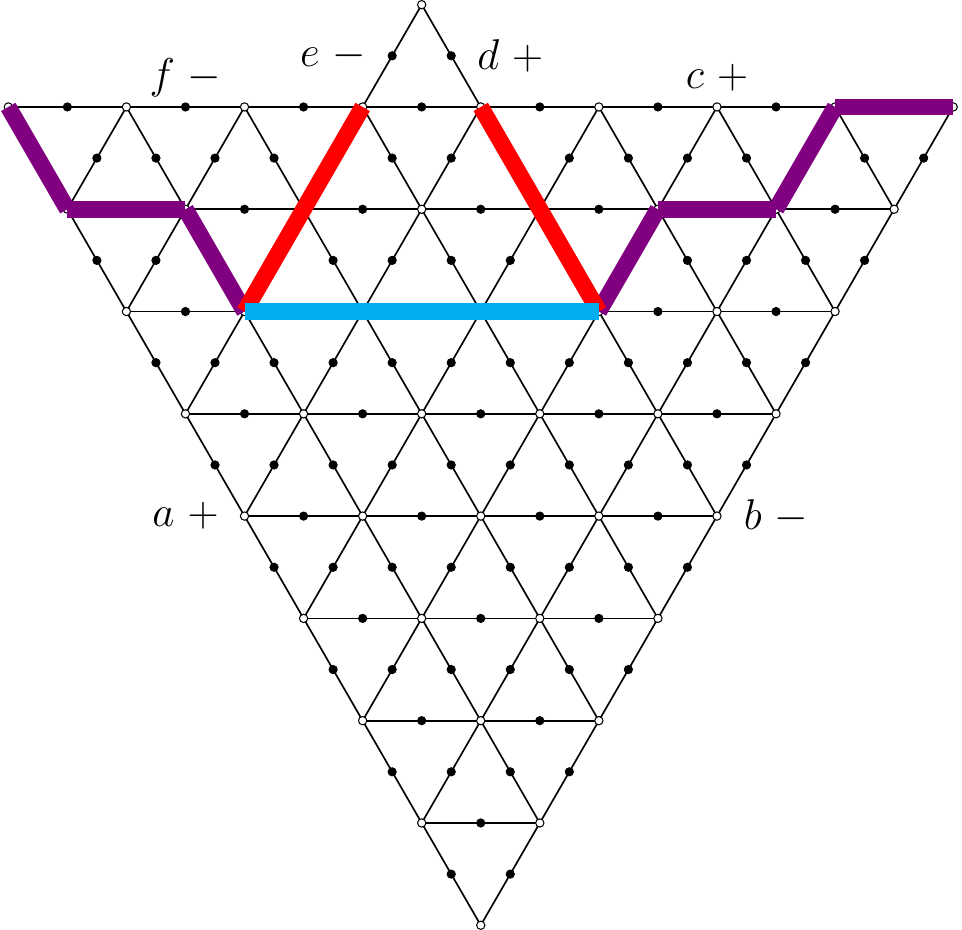}
            \caption{Four sectors}
            \label{subfig:044Divide}
        \end{subfigure}
     \quad
        \begin{subfigure}[b]{0.4\textwidth}
            \centering
            \includegraphics[width = \textwidth]{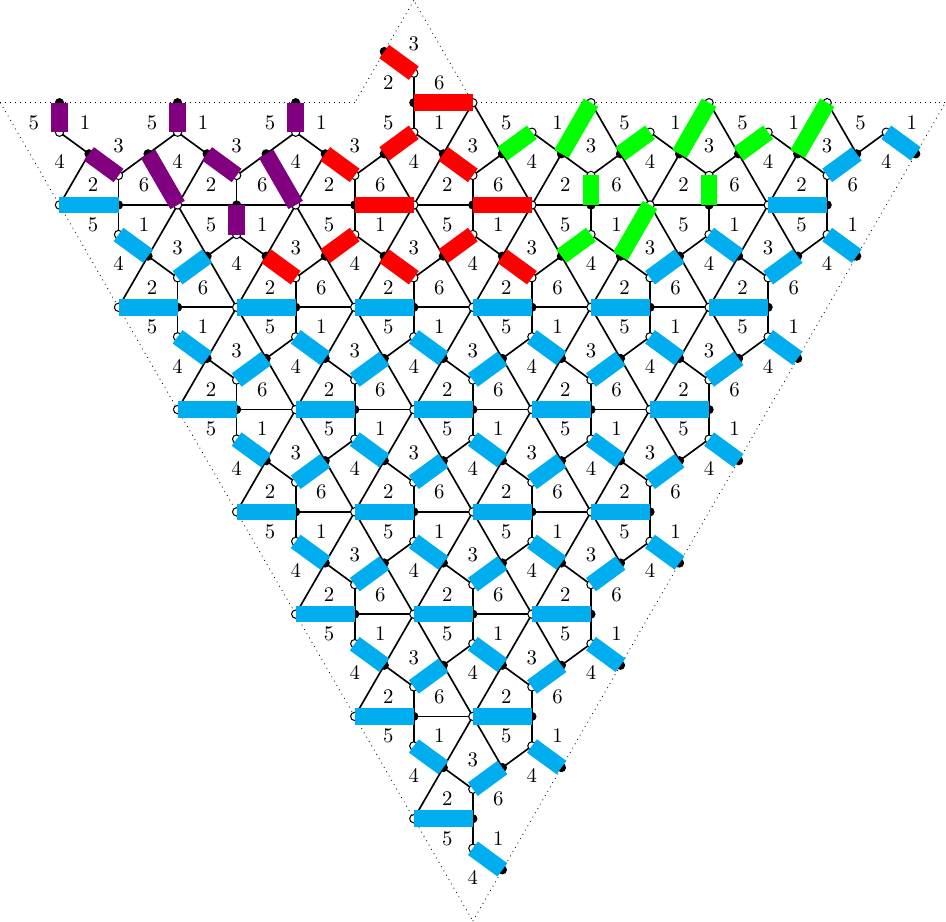}
            \caption{Minimal matching}
            \label{subfig:044Matching}
        \end{subfigure}
        \caption{Four sectors and minimal matching of Castle $\mathcal{C}_{0,4,4}$}
        \label{fig:044Example}
    \end{figure}

\newpage

    \textbf{Region 3:} $k\geq 1; i,j < k-1 < i+j$. This is the blue region in the first quadrant in Figure \ref{fig:allRegionPos}.

    \begin{figure}[h!]
    \centering
        \begin{subfigure}[b]{0.4\textwidth}
            \centering
            \includegraphics[width = \textwidth]{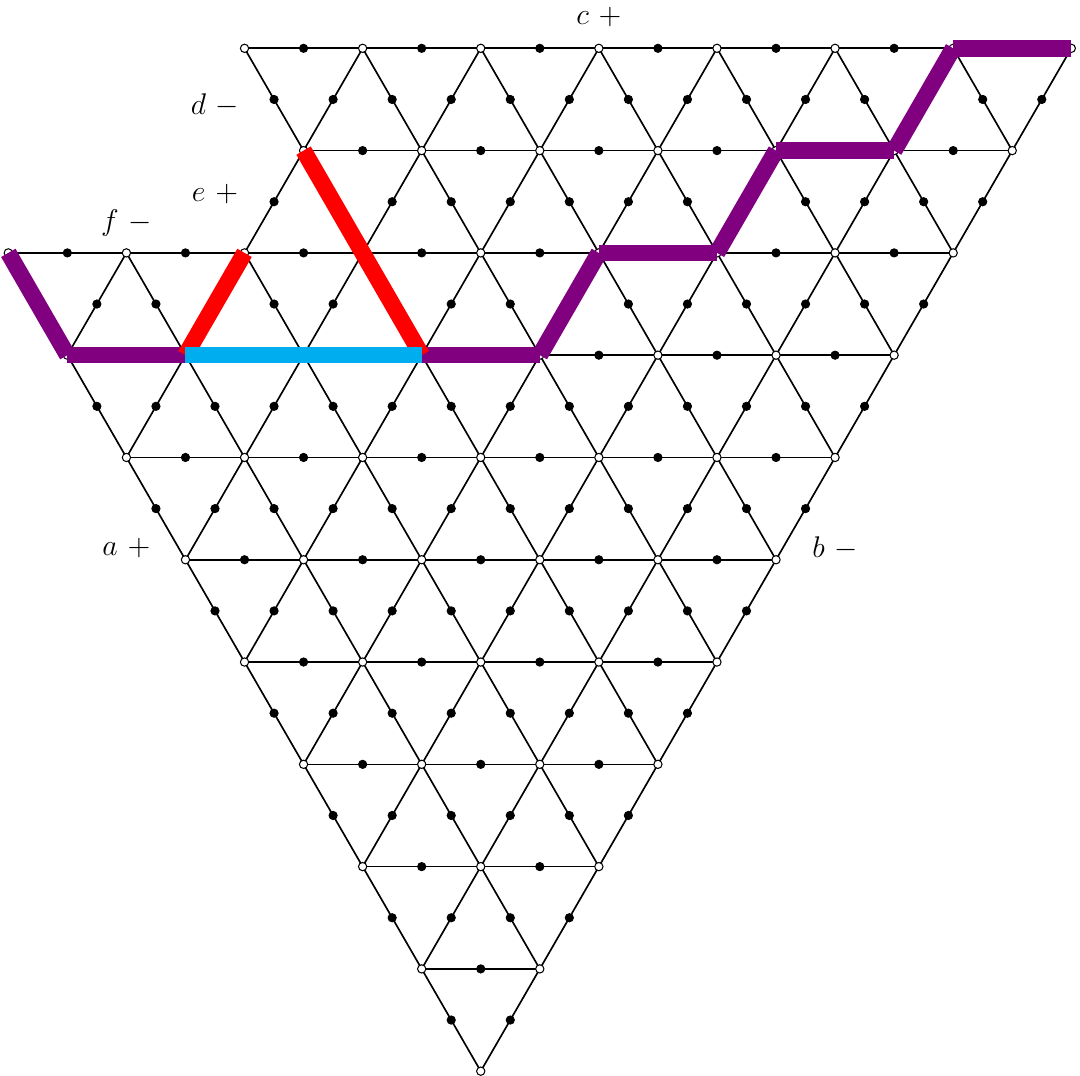}
            \caption{Four sectors}
            \label{subfig:235Divide}
        \end{subfigure}
     \quad
        \begin{subfigure}[b]{0.4\textwidth}
            \centering
            \includegraphics[width = \textwidth]{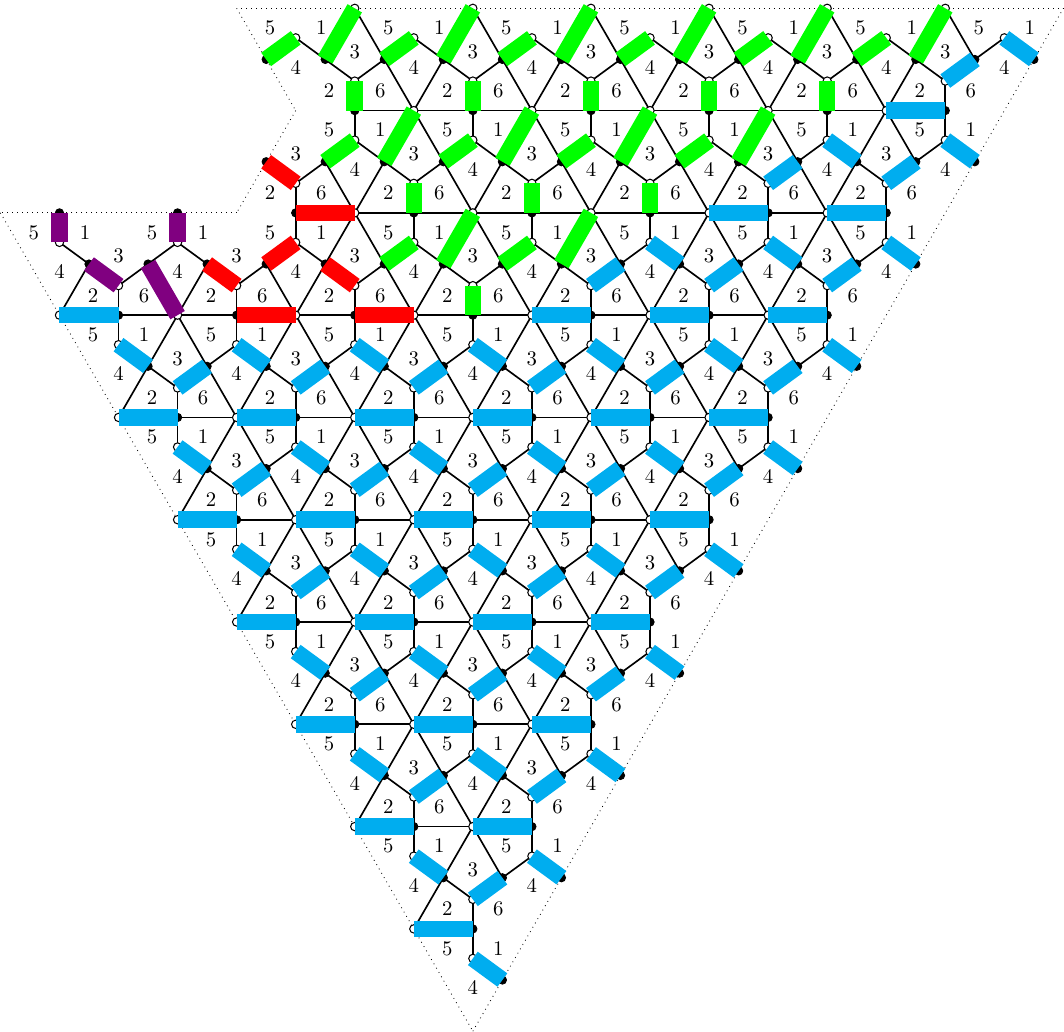}
            \caption{Minimal matching}
            \label{subfig:235Matching}
        \end{subfigure}
        \caption{Four sectors and minimal matching of Castle $\mathcal{C}_{2,3,5}$}
        \label{fig:235Example}
    \end{figure}

    \textbf{Region $1 \cap 2$:} $k\geq 1; k - 1 = i < j$. This is boundary between region 1 and 2 above.

    \begin{figure}[h!]
    \centering
        \begin{subfigure}[b]{0.4\textwidth}
            \centering
            \includegraphics[width = \textwidth]{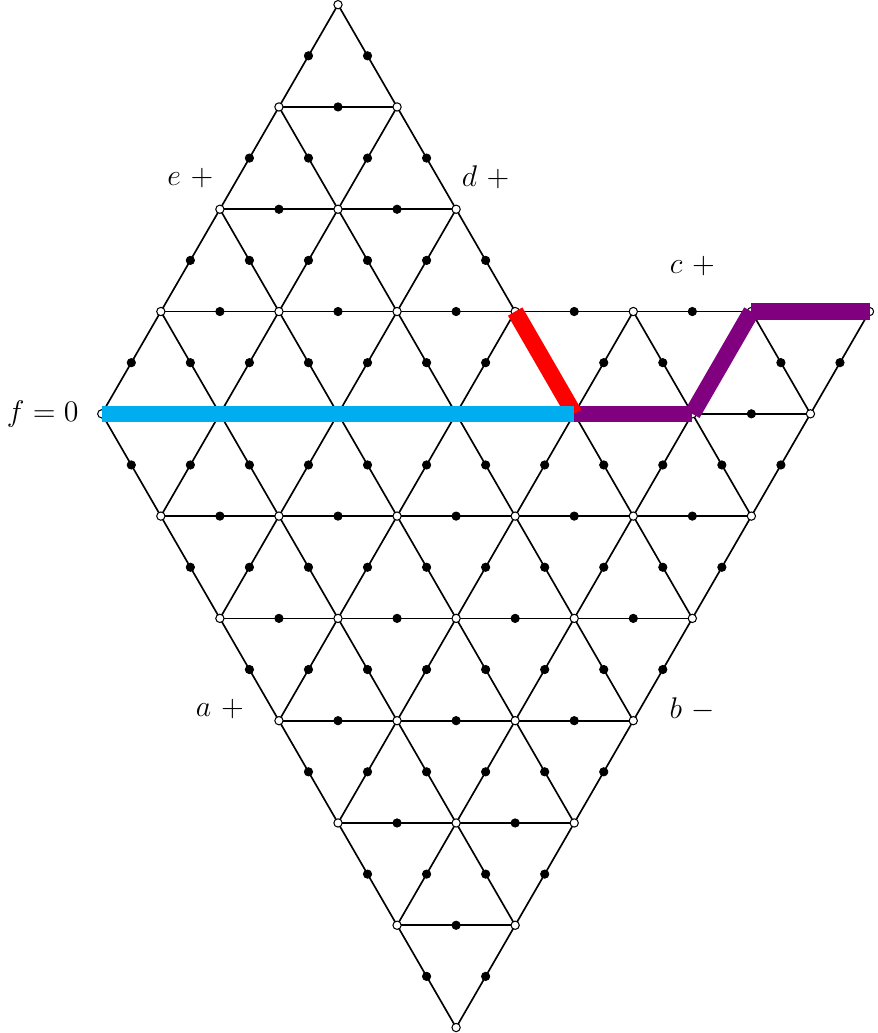}
            \caption{Four sectors}
            \label{subfig:142Divide}
        \end{subfigure}
     \quad
        \begin{subfigure}[b]{0.4\textwidth}
            \centering
            \includegraphics[width = \textwidth]{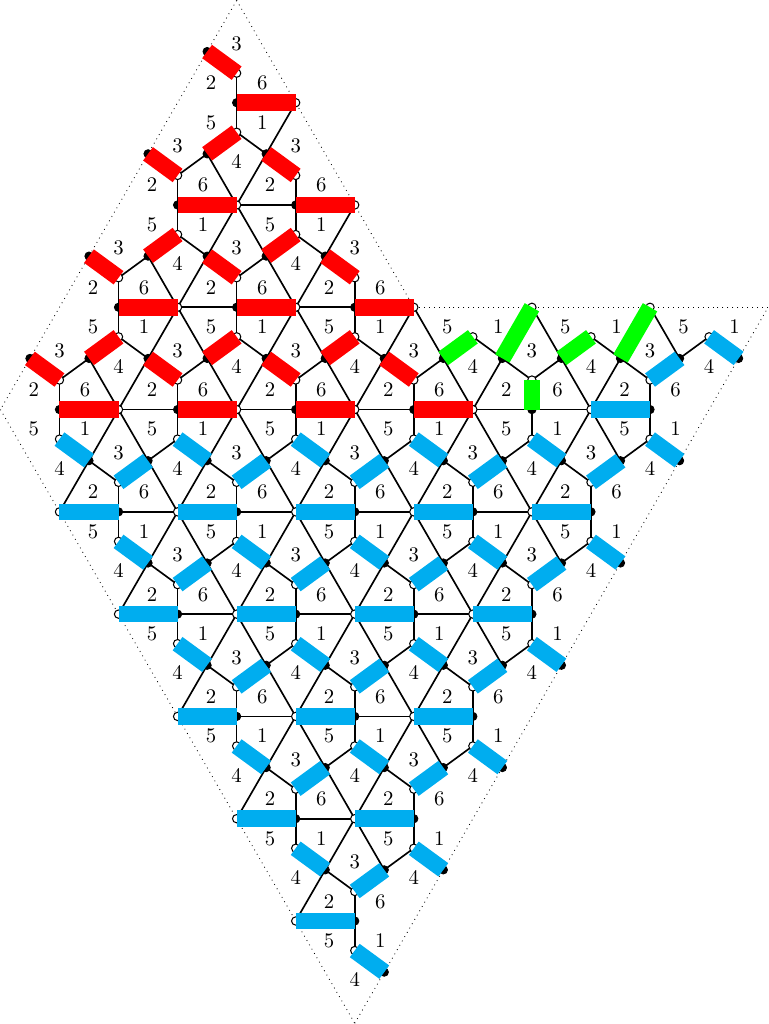}
            \caption{Minimal matching}
            \label{subfig:142Matching}
        \end{subfigure}
        \caption{Four sectors and minimal matching of Castle $\mathcal{C}_{1,4,2}$}
        \label{fig:142Example}
    \end{figure}

\newpage

    \textbf{Region $2 \cap 3$:} $k\geq 1; 0 < i < k-1 = j$. This is boundary between region 2 and 3 above.

    \begin{figure}[h!]
    \centering
        \begin{subfigure}[b]{0.4\textwidth}
            \centering
            \includegraphics[width = \textwidth]{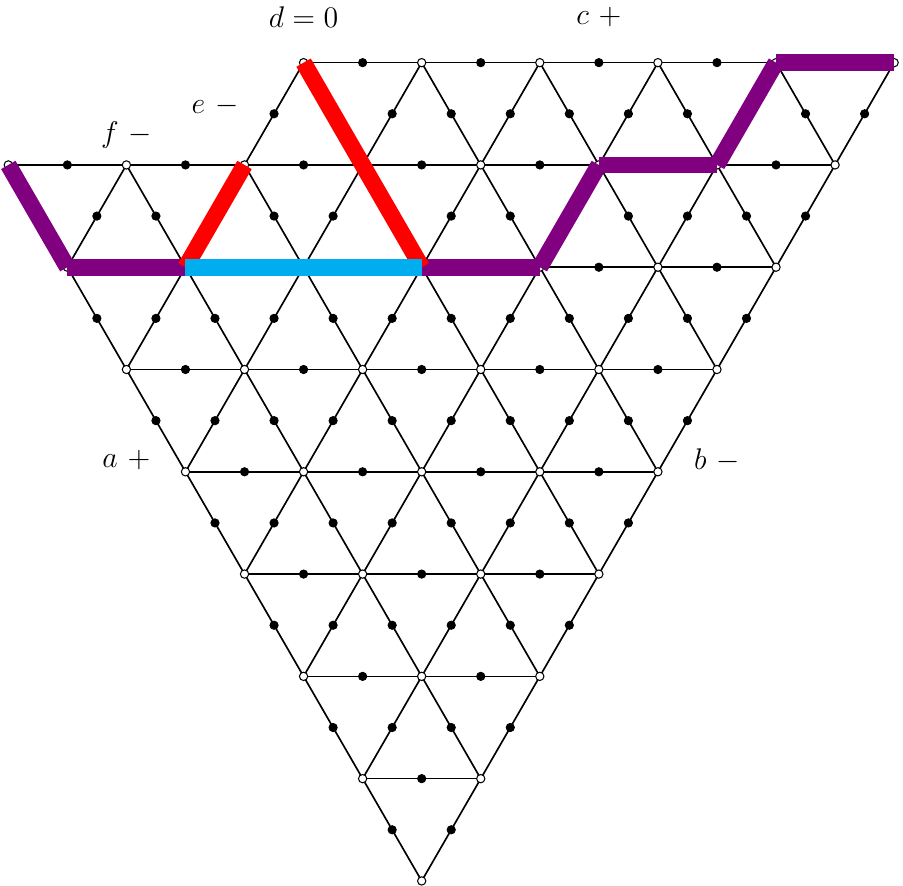}
            \caption{Four sectors}
            \label{subfig:134Divide}
        \end{subfigure}
     \quad
        \begin{subfigure}[b]{0.4\textwidth}
            \centering
            \includegraphics[width = \textwidth]{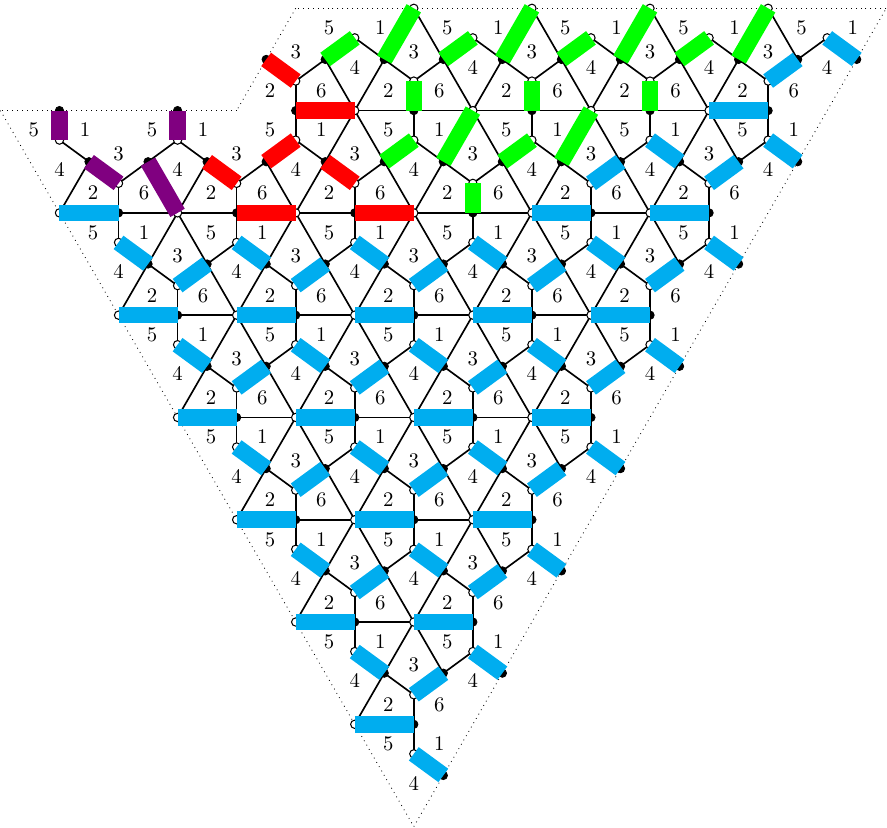}
            \caption{Minimal matching}
            \label{subfig:134Matching}
        \end{subfigure}
        \caption{Four sectors and minimal matching of Castle $\mathcal{C}_{1,3,4}$}
        \label{fig:134Example}
    \end{figure}

    \textbf{Region $1 \cap 2 \cap 3$:} $k\geq 1; i = j = k-1$. 
    This is the point where regions 1, 2, and 3, as above, intersect.

    \begin{figure}[h!]
    \centering
        \begin{subfigure}[b]{0.4\textwidth}
            \centering
            \includegraphics[width = \textwidth]{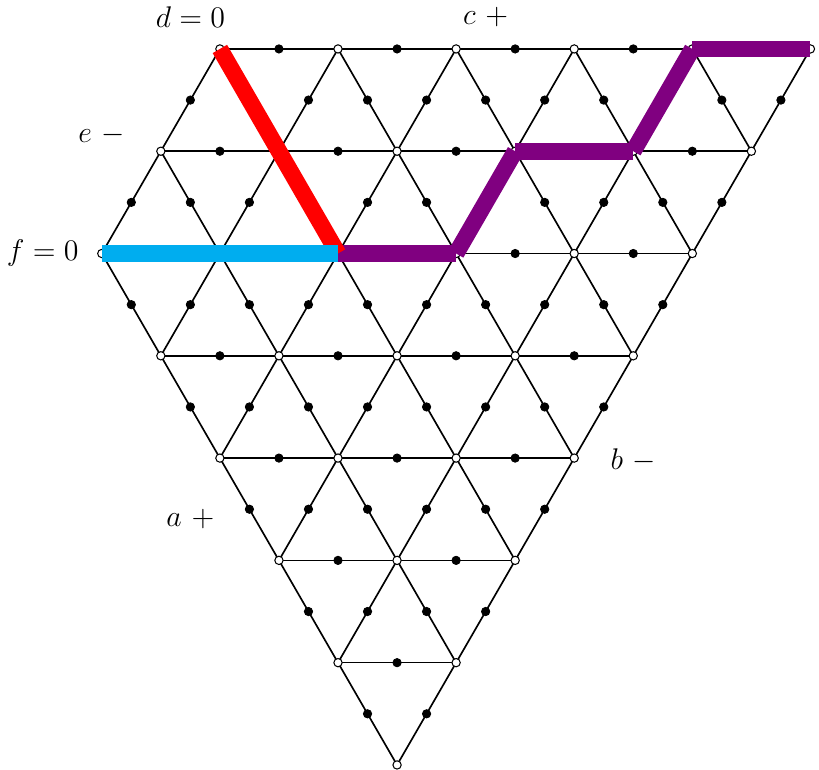}
            \caption{Four sectors}
            \label{subfig:223Divide}
        \end{subfigure}
     \quad
        \begin{subfigure}[b]{0.4\textwidth}
            \centering
            \includegraphics[width = \textwidth]{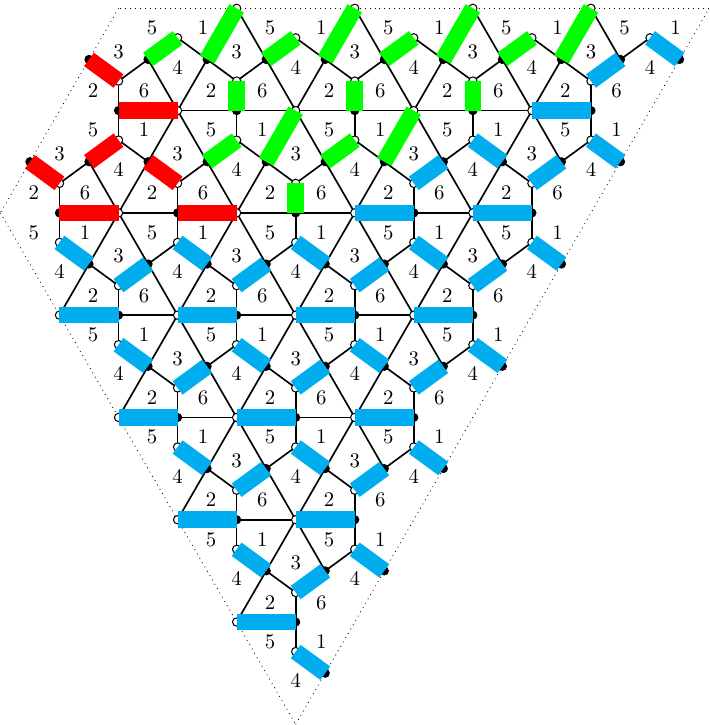}
            \caption{Minimal matching}
            \label{subfig:223Matching}
        \end{subfigure}
        \caption{Four sectors and minimal matching of Castle $\mathcal{C}_{2,2,3}$}
        \label{fig:223Example}
    \end{figure}

\newpage

\section{Aztec Dragons}\label{append:dragons}
    
    A special family of Aztec Castles is the family of Aztec Dragons, which are Aztec Castles with $i\in\{-1,0\},k\in\{0,1\}$, and $j \geq 0$. The weights of Aztec Dragons are the cluster variables of the dP3 quiver after mutations $\tau_1\tau_2\tau_3\tau_1\tau_2\tau_3\ldots$. Specifically, the four types of Aztec Dragons are

    \begin{itemize}
        \item $D_n = C_{0,n,1}$,
        \item $D'_n = C_{0,n,0}$,
        \item $D_{n+1/2} = C_{-1,n+1,0}$, and
        \item $D'_{n+1/2} = C_{-1,n+1,1}$.
    \end{itemize}
    
    Note that $D_n$ lies in Region $1\cap 2$ in Appendix \ref{append:examples}. $D'_n$ also lies in the same region but with $k\leq 0$. $D'_{n+1/2}$ lies in the boundary between Region 2 and 1' in Section \ref{subsec:zero-line}, and $D_{n+1/2}$ lies in the same region but with $k\leq 0$. Hence, $D'_n$ is a $180^{\circ}$ rotation of $D_n$, and $D'_{n+1/2}$ is a $180^{\circ}$ rotation of $D_{n+1/2}$. In this appendix, we show examples of minimal matchings of Aztec Dragons of the four types.

    \textbf{Type 1:} $D_n$. This is an example of a graph from Region $1 \cap 2$ with $f = 0$. Moreover, side $c$ always has length $i + k = 1$, so the staircase is just a single step which blends into the zero line and the relevant straight lines have length zero. In addition, since the lengths of $c$ and $f$ stay the same, the vertices on the zero line are always covered by the covering of the upper sector.

    \begin{figure}[h!]
    \centering
        \begin{subfigure}[b]{0.3\textwidth}
            \centering
            \includegraphics[width = \textwidth]{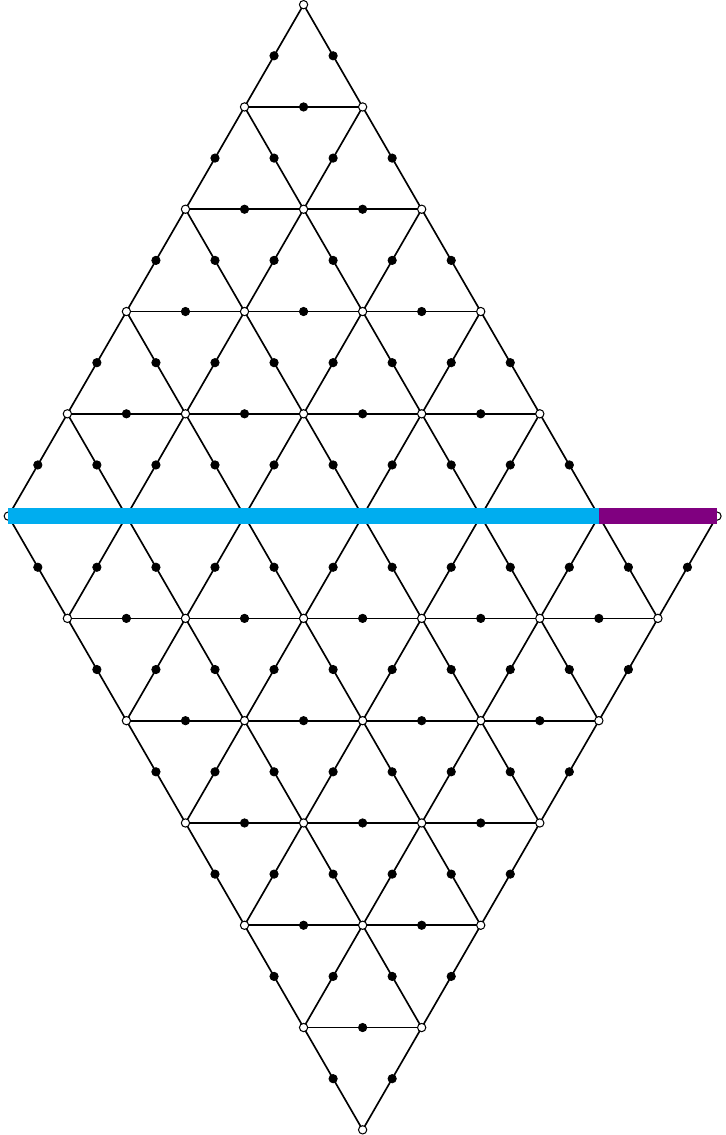}
            \caption{Four sectors}
            \label{subfig:051Divide}
        \end{subfigure}
     \quad\quad\quad
        \begin{subfigure}[b]{0.3\textwidth}
            \centering
            \includegraphics[width = \textwidth]{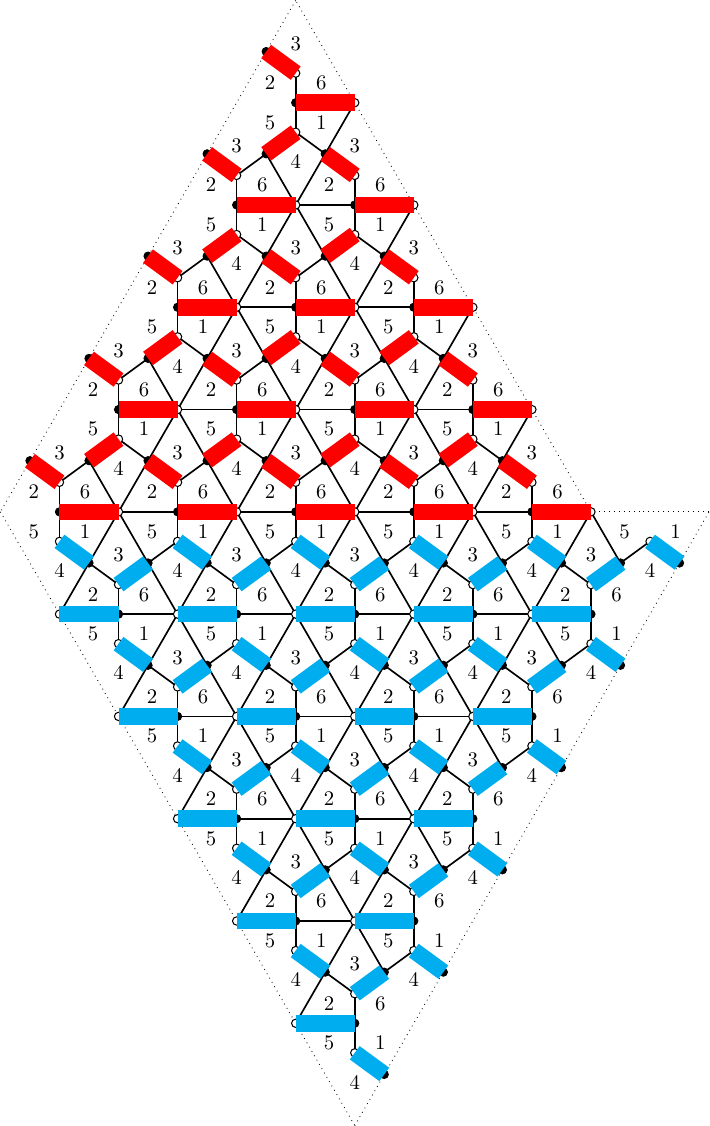}
            \caption{Minimal matching}
            \label{subfig:051Matching}
        \end{subfigure}
        \caption{Four sectors and minimal matching of Dragon $D_5 = \mathcal{C}_{0,5,1}$}
        \label{fig:051Example}
    \end{figure}

    \newpage

    \textbf{Type 2:} $D'_n$. Recall that $D'_n$ is a $180^{\circ}$ rotation of $D_n$, and one can check that Figure \ref{fig:050Example} is a $180^{\circ}$ rotation of Figure \ref{fig:051Example}.

    \begin{figure}[h!]
    \centering
        \begin{subfigure}[b]{0.3\textwidth}
            \centering
            \includegraphics[width = \textwidth]{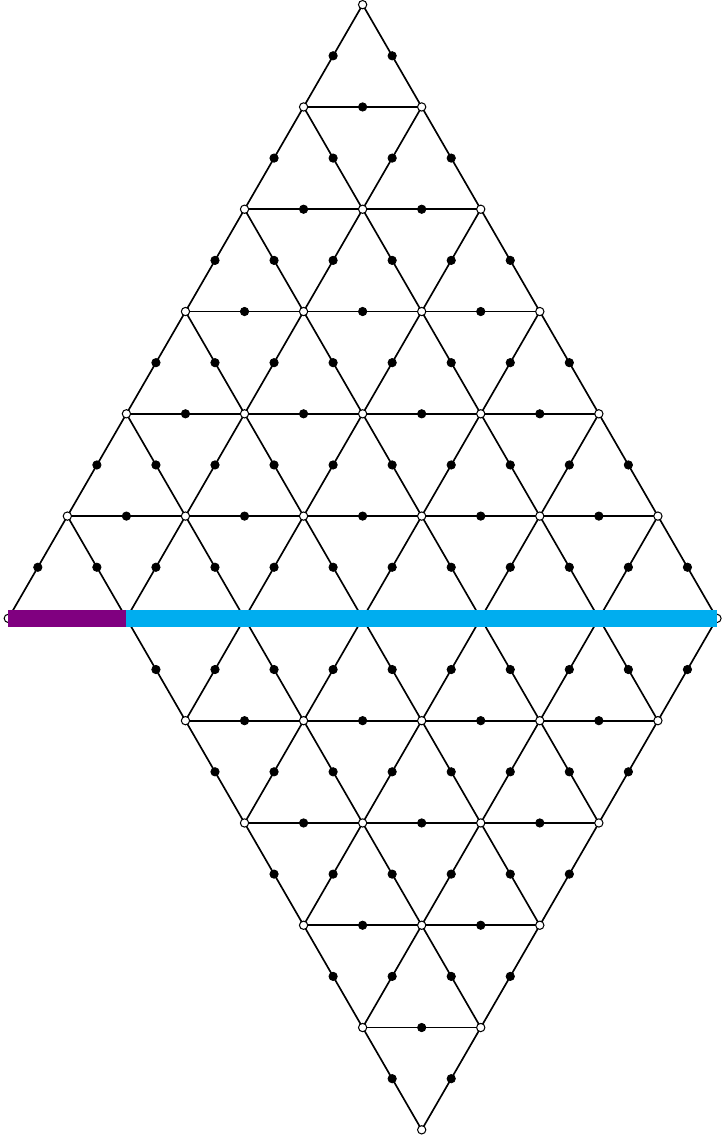}
            \caption{Four sectors}
            \label{subfig:050Divide}
        \end{subfigure}
     \quad\quad\quad
        \begin{subfigure}[b]{0.3\textwidth}
            \centering
            \includegraphics[width = \textwidth]{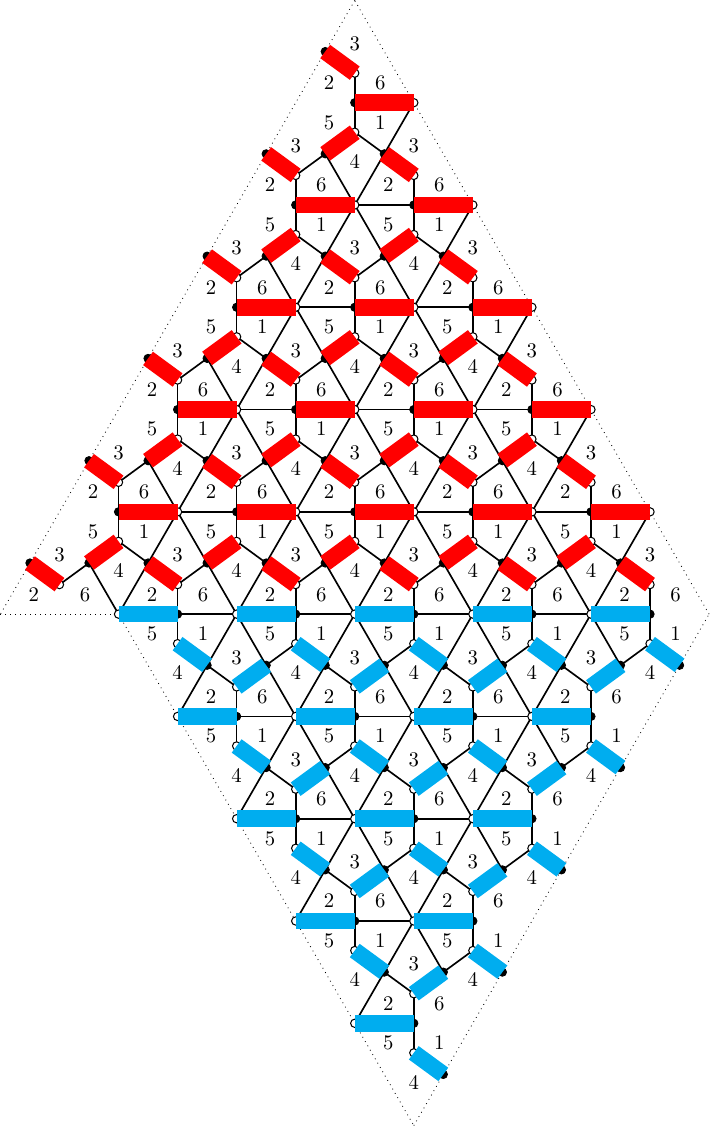}
            \caption{Minimal matching}
            \label{subfig:050Matching}
        \end{subfigure}
        \caption{Four sectors and minimal matching of Dragon $D'_5 = \mathcal{C}_{0,5,0}$}
        \label{fig:050Example}
    \end{figure}

    \textbf{Type 3:} $D_{n+1/2}$. Here, we still have $f = 0$ and $c = -1$. Hence, the staircase and straight line both have length $1$. Again, since the lengths of $c$ and $f$ stay the same, the vertices on the zero line are always covered by the covering of the upper sector.

    \begin{figure}[h!]
    \centering
        \begin{subfigure}[b]{0.3\textwidth}
            \centering
            \includegraphics[width = \textwidth]{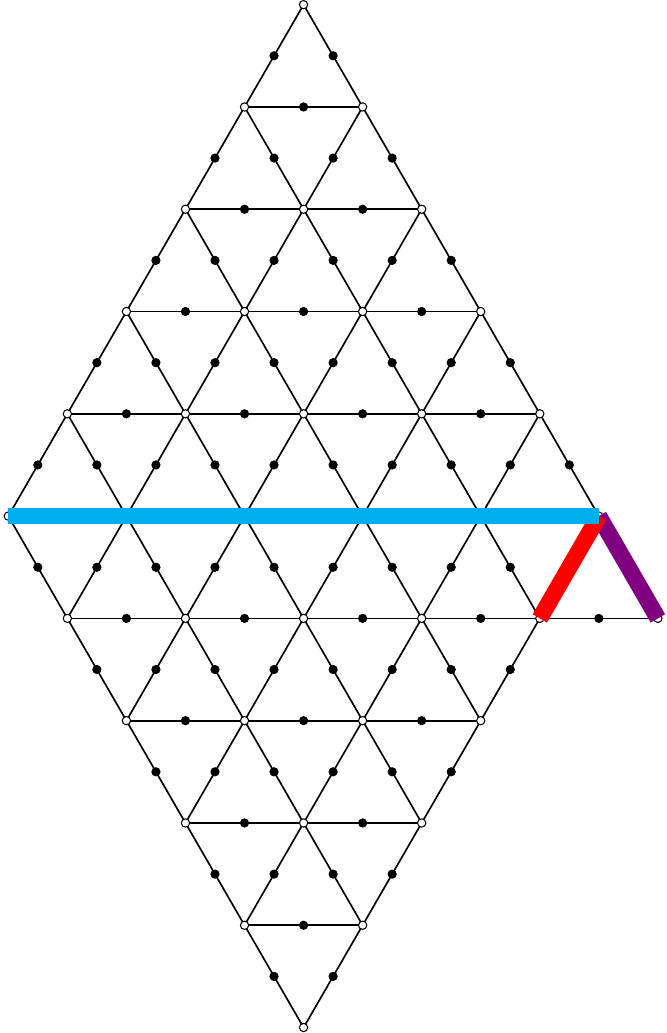}
            \caption{Four sectors}
            \label{subfig:n150Divide}
        \end{subfigure}
     \quad\quad\quad
        \begin{subfigure}[b]{0.3\textwidth}
            \centering
            \includegraphics[width = \textwidth]{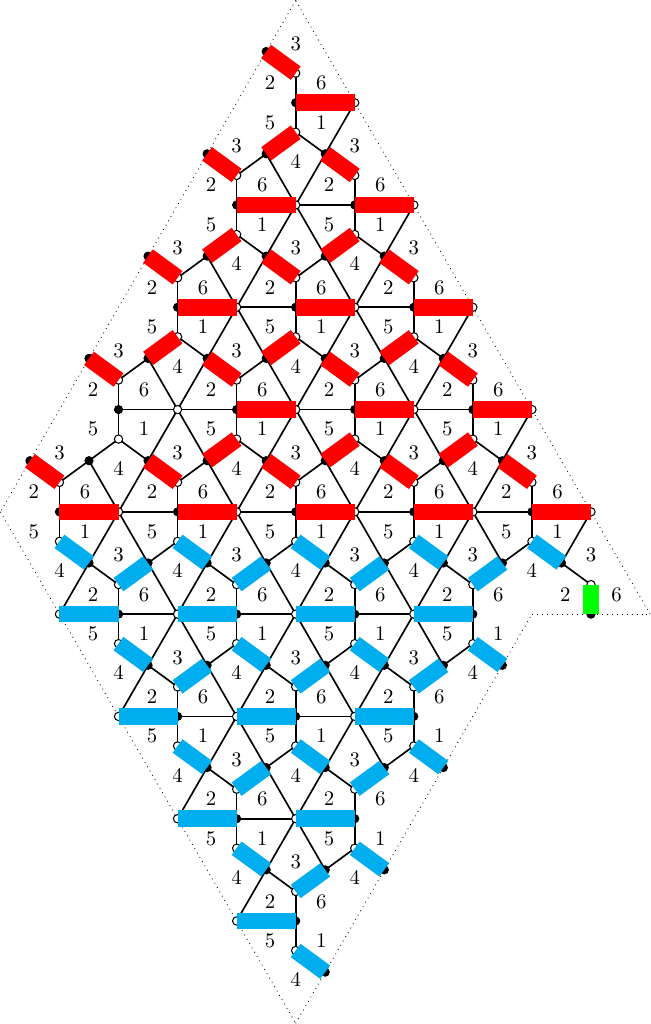}
            \caption{Minimal matching}
            \label{subfig:n150Matching}
        \end{subfigure}
        \caption{Four sectors and minimal matching of Dragon $D_{4+1/2} = \mathcal{C}_{-1,5,0}$}
        \label{fig:n150Example}
    \end{figure}

    \newpage

    \textbf{Type 4:} $D'_{n+1/2}$. Recall that $D'_{n+1/2}$ is a $180^{\circ}$ rotation of $D_{n+1/2}$, and one can check that Figure \ref{fig:n151Example} is a $180^{\circ}$ rotation of Figure \ref{fig:n150Example}.

    \begin{figure}[h!]
    \centering
        \begin{subfigure}[b]{0.3\textwidth}
            \centering
            \includegraphics[width = \textwidth]{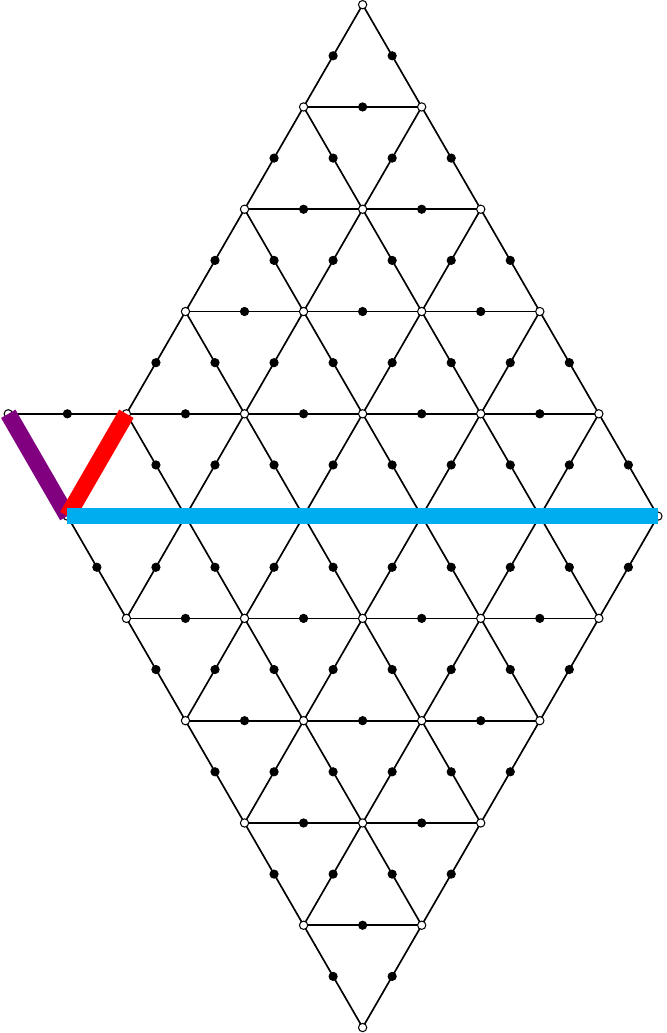}
            \caption{Four sectors}
            \label{subfig:n151Divide}
        \end{subfigure}
     \quad\quad\quad
        \begin{subfigure}[b]{0.3\textwidth}
            \centering
            \includegraphics[width = \textwidth]{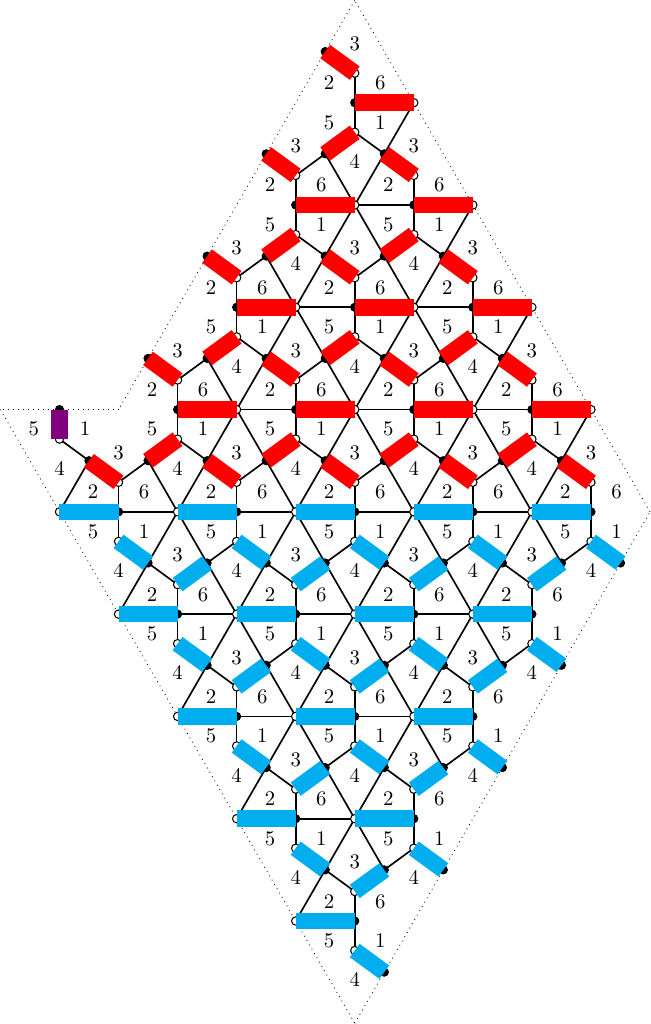}
            \caption{Minimal matching}
            \label{subfig:n151Matching}
        \end{subfigure}
        \caption{Four sectors and minimal matching of Dragon $D'_{4+1/2} = \mathcal{C}_{-1,5,1}$}
        \label{fig:n151Example}
    \end{figure}

\end{document}